\documentclass[12pt,reqno]{amsart}

\usepackage{mathtools}
\usepackage{amsmath}
\usepackage{amssymb}
\usepackage{amsthm}
\usepackage{amscd}
\usepackage[mathscr]{eucal}
\usepackage{bbm}
\usepackage{color}
\usepackage[english=american]{csquotes}
\usepackage{bm}
\usepackage{color}
\usepackage{calc}
\usepackage{enumerate}
\usepackage{esint}
\usepackage{mathabx}
\usepackage[pagebackref]{hyperref}
\usepackage[normalem]{ulem}

\usepackage{hyperref}
\hypersetup{
	colorlinks=true,
	citecolor=blue,
	linkcolor=blue,
	filecolor=black,      
	urlcolor=black,
}
\makeatother
\usepackage[all]{xy}
\usepackage[textwidth=8em,textsize=small]{todonotes}
\newcommand{\mres}{\mathbin{\vrule height 1.6ex depth 0pt width
		0.13ex\vrule height 0.13ex depth 0pt width 1.3ex}}

\newcommand{\Acal}{\mathcal{A}}

\newcommand{\Fcal}{\mathcal{F}}

\newcommand{\Hcal}{\mathcal{H}}

\newcommand{\Kcal}{\mathcal{K}}
\newcommand{\Lcal}{\mathcal{L}}
\newcommand{\Mcal}{\mathcal{M}}

\newcommand{\Scal}{\mathcal{S}}

\newcommand{\Vcal}{\mathcal{V}}

\newcommand{\Afrak}{\mathfrak{A}}

\newcommand{\Dfrak}{\mathfrak{D}}

\newcommand{\Ffrak}{\mathfrak{F}}
\newcommand{\Gfrak}{\mathfrak{G}}

\newcommand{\Ksf}{\mathsf{K}}

\newcommand{\R}{\mathbb{R}}
\newcommand{\Sbf}{\mathbf{S}}

\newcommand{\Abb}{\mathbb{A}}

\newcommand{\Sbb}{\mathbb{S}}
\newcommand{\Tbb}{\mathbb{T}}

\newcommand{\Zbb}{\mathbb{Z}}

\DeclareMathOperator{\pv}{p.\!v.}

\DeclareMathOperator{\diverg}{div}

\DeclareMathOperator{\dist}{dist}
\DeclareMathOperator{\rank}{rank}

\newcommand{\set}[2]{\left\{\, #1 \ \textup{{:}}\ #2 \,\right\}}

\newcommand{\dpr}[1]{\langle #1 \rangle}

\newcommand{\dd}{\;\mathrm{d}}

\newcommand{\loc}{\mathrm{loc}}

\newcommand{\toweakstar}{\overset{*}\rightharpoonup}

\newcommand{\embed}{\hookrightarrow}
\newcommand{\cembed}{\overset{c}{\embed}}

\newcommand{\sbullet}{\begin{picture}(1,1)(-0.5,-2)\circle*{2}\end{picture}}
\newcommand{\frarg}{\,\sbullet\,}

\newcommand{\eps}{\varepsilon}

\newtheoremstyle{thmlemcorr}{15pt}{15pt}{\itshape}{15pt}{\bfseries}{.}{5pt}{{\thmname{#1}\thmnumber{ #2}\thmnote{ (#3)}}}
\newtheoremstyle{thmlemcorr*}{15pt}{15pt}{\itshape}{15pt}{\bfseries}{.}\newline{{\thmname{#1}\thmnumber{ #2}\thmnote{ (#3)}}}
\newtheoremstyle{defi}{15pt}{15pt}{}{15pt}{\bfseries}{.}{5pt}{{\thmname{#1}\thmnumber{ #2}\thmnote{ (#3)}}}
\newtheoremstyle{remexample}{15pt}{15pt}{}{15pt}{\itshape}{.}{5pt}{{\thmname{#1}\thmnumber{ #2}\thmnote{ (#3)}}}
\newtheoremstyle{ass}{10pt}{10pt}{}{}{\bfseries}{.}{10pt}{{\thmname{#1}\thmnumber{ }\thmnote{ (#3)}}}

\theoremstyle{defi}
\newtheorem{definition}{Definition}[section]

\theoremstyle{thmlemcorr}
\newtheorem{theorem}{Theorem}
\newtheorem{lemma}[definition]{Lemma}
\newtheorem{corollary}[definition]{Corollary}
\newtheorem{proposition}[definition]{Proposition}

\newtheorem{conjecture}{Conjecture}

\newtheorem{thmx}{Theorem}

\theoremstyle{thmlemcorr*}
\newtheorem{theorem*}[section]{Theorem}
\newtheorem{lemma*}{Lemma}
\newtheorem{corollary*}{Corollary}
\newtheorem{proposition*}{Proposition}
\newtheorem{problem*}{Problem}
\newtheorem{conjecture*}{Conjecture}

\theoremstyle{remexample}
\newtheorem{example}[theorem]{Example}

\newtheorem{remark}{Remark}

\theoremstyle{ass}

\date{}

\title[Localization of nonlocal linear PDEs]{Functional and variational \\ aspects of nonlocal operators \\ associated with linear PDEs}
\author{
	Adolfo Arroyo-Rabasa}
	\address{Adolfo Arroyo-Rabasa, 
		Institut f\"ur Angewandte Mathematik, Universit\"at Bonn, 53115 Bonn, Germany.}
	\email{rabasa@iam.uni-bonn.de}
\date{}

\begin{document}
	\maketitle

	\begin{abstract}
		We introduce a general difference quotient representation for non-local operators associated with a first-order linear operator. We establish new local to non-local estimates and strong localization principles in various spaces of functions, measures and distributions, which fully generalize those known for gradients. Under suitable assumptions, we also establish the invariance of quasiconvexity within the proposed local-nonlocal setting. Applications to the fine properties of $\mathcal A$-gradient measures are further discussed.\\

		\vspace{4pt}
		
		\vspace{5pt}
	\end{abstract}
\setcounter{tocdepth}{1}
		\tableofcontents

    \maketitle

\section{Introduction}We consider a general constant coefficient first-order homogeneous differential operator $\Acal : \mathscr D'(\R^n;V) \to \mathscr D'(\R^n;W)$, where $V,W$ are finite-dimensional inner product spaces (identified with their own duals) and $\mathscr D' = (C^\infty_c)'$ denotes the space of distributions.  Such operators can be written in the form
\begin{equation}\label{eq:A}
	\Acal = \sum_{i = 1}^n A_i \partial_i\;, \qquad A_i \in \mathrm{Hom}(V,W)\,.
\end{equation}
We recall that, in this form and up to a complex constant, the principal symbol associated with $\Acal$ is given by the $\mathrm{Hom}(V,W)$-valued linear polynomial 
\[
	\Abb(\xi) \coloneqq \sum_{i = 1}^n A_i \xi_i\,,\qquad \xi = (\xi_1,\dots,\xi_n) \in \R^n.
\] 
Given a radially symmetric weight $\rho : \R^n \to [0,\infty]$, we
introduce an associated non-local gradient  $\Afrak_\rho$, which, on sufficiently regular maps $u : \R^n \to V$ is given by the principal value integral  
\[
	\Afrak_\rho u(x) \coloneqq \pv \left( n \int \Omega_\Abb(\xi)\left[\frac{u(x + \xi) - u(x)}{|\xi|}\right]\, \rho(\xi) \, d\xi\, \right)\,.
\]
Here, 
\[
	\Omega_\Abb(\xi) \coloneqq \frac{\Abb(\xi)}{|\xi|}
\]
is the zero-homogeneous profile of the principal symbol. 

From a wide perspective, the motivation for the present work is to study function spaces associated with general linear differential operators. 
This area of research started its development with the introduction of the \emph{compensated compactness theory} (e.g.,~ \cite{me,me2,advances,Davoli,FM99}), which served as a tool to study numerous models in continuum mechanics and materials science. Since then, there has been a profound interest in understanding general differential constraints as in~\eqref{eq:A}.
It is also worth mentioning that the $\Acal$-framework has served as a bedrock for profound developments in \emph{PDE theory and geometric measure theory}~\cite{gafa, de2016structure,FG, VS_13}. 
Departing from this premise, the primary goal of this paper is to introduce a transparent methodology that allows us to generalize the well-established theory for non-local gradients to the $\Acal$-gradient framework.

\subsection*{Previous work}
It is well-known that a map $u \in L^1(\R^n)$ belongs to the Sobolev space $u \in W^{1,p}(\R^n)$ if and only if 
\begin{equation}\label{eq:pointwiseD}
	\liminf_{h \to 0^+} \|\delta_{h\omega} u \|_{L^p} < \infty \quad \text{for all directions $\omega \in \Sbb^{n-1}$},
\end{equation}
where
\[
	\delta_{h\xi} u(x)\coloneqq \frac{u(x + h\xi) - u(x)}{h}
\] 
is the $h$-scale point-wise difference quotient of $u$ on the direction $\xi$. If~\eqref{eq:pointwiseD} holds, then the full limit exists and 
\[
	\delta_{h\xi} u \longrightarrow Du\cdot\xi
\]
strongly in $L^p$ as $h \to^+ 0$\,.
In this context, {Bourgain, Brezis and Mironescu}~\cite{BBM_01_Another_look} established the following difference quotient representation of Sobolev spaces: If $u \in W^{1,p}(\R^n)$, then 
\[
	\lim_{\eps \to 0^+} \iint \frac{|u(x+h) - u(x)|^p}{|h|^p} \, \rho_\eps(h) \dd h \dd x= K_{p,n} \int |\nabla u|^p,
\]
where $K_{p,n}$ is a constant depending solely on $p$ and the spatial dimension $n$. Here,
$\{\rho_\eps\}_{\eps > 0}\subset L^1(\R^n)$ is a family of probability non-negative radial functions, that is,   
\begin{equation}\label{eq:eps1}
\|\rho_\eps\|_{L^1(\R^n)}  = 1.
\end{equation}
Further, it is assumed that this family approximates the Dirac mass at zero in the sense that
\begin{equation}\label{eq:eps}
\lim_{\eps \to 0^+} \|\rho_\eps\|_{L^1(\R^n \setminus B_\delta)} 
= 0 \qquad \text{for all} \; \delta > 0.
\end{equation}
This statement has a converse, which provides a \emph{difference quotient criterion} for Sobolev maps in the range $1 < p < \infty$. The case for $p=1$ is special in that concentrations can occur. In this vein, {D\'avila}~\cite{Davila_02} showed that a related criterion holds for $BV(\R^n)$ when $p = 1$, by replacing $\int |\nabla u|$ with the total variation $|Du|(\R^n)$. 
Similar difference quotient representations have been obtained for a norm of the symmetric-gradient operator\footnote{The case $1 < p < \infty$ was established by {Mengesha}~\cite{MengeshaBD} on $(W^{1,p})^n$, and the case $p=1$ by {Arroyo-Rabasa and Bonicatto}~\cite{BD} on $BD$.}\textsuperscript{,}\footnote{Presently, a full characterization of operators satisfying a difference quotient \emph{Bourgain--Brezis--Mironsecu (BBM) criterion} is being developed by Arroyo-Rabasa, Buccheri and Van Schaftingen~\cite{bbm}.}
\[
	Eu = \frac 12(Du + Du^t)\,, \qquad u : \R^n \to \R^n\,.
\]
Notice that, in general, these types of BBM criteria only convey the approximation of a gradient norm, but themselves \emph{do not define a non-local operator}.

Motivated by the ``energy criterion'' and by the results contained in~\cite{Du_13} (where the first rigorous non-local vector calculus was developed for the gradient, curl, and divergence operators in Hilbert spaces), {Mengesha and Spector}~\cite{MS} (see also~\cite{Ponce2}) proved that the non-local gradient 
\begin{align*}
	\Dfrak_{\rho_\eps} u & \coloneqq n \left(\pv \int  \delta_{\xi} u \, \frac{\xi}{|\xi|} \, \rho_\eps(\xi)  \dd \xi\right) 
\end{align*}
``localizes'' towards the classical gradient $D$ in a related fashion to the BBM characterization of Sobolev space above (except that it allows for cancellations inside the integral).\footnote{Mengesha and Spector denote the non-local gradient operator by $\Gfrak$.} More precisely,  they showed that if $u$ belongs to a certain function space $(X,\tau)$, then 
\begin{equation}\label{D}
	\Dfrak_{\rho_\eps} u \stackrel{\tau}\longrightarrow Du \qquad \text{as $\eps \to 0$\,.}
\end{equation}
This framework was introduced as a mathematical tool to develop peridynamics models of continuum mechanics, over nonlocal integral operators describing physically relevant laws and quantities (see \cite{Mengesha_2,DuTian,Duse,Elia,Elia2,Friedrich,kruzik1,DM_2,DM_16,MS} and references therein).
These localization results for the non-local gradient were established for various relevant topologies, including the Sobolev space $X = W^{1,p}$, and the space of functions with bounded variation $X = {BV}$ endowed with the topology associated with the \emph{area-convergence} of measures.  

Later on, Du and Mengesha (see \cite{DM_16}) used classical Sobolev-type slicing arguments to observe the validity  of \emph{soft estimates}
\begin{equation}\label{eq:DM1}
	\int \left| \int \dpr{\mathsf K(\xi),u(x), \frac{\xi}{|\xi|}, \nabla \psi(x)} \, \rho(\xi) \, d \xi \right|^p \, dx \lesssim \|\rho\|_{L^1}\|Du\|_{L^p}^p\,,
\end{equation}
where $\Ksf$ is a zero-homogeneous third-order tensor and (with the double index summation convention) 
\[
	\dpr{\Ksf,u,v,w} \coloneqq \Ksf_{ijk}u_iv_jw_k\,.
\]
There, the authors also established
 \emph{weak-convergence} 
 localization principles for distributional {integro-differential operators} of the form
\begin{equation}\label{eq:DM}
	\dpr{D^{\mathsf K}u,\psi} = \int \int \dpr{\mathsf K(\xi),u(x), \frac{\xi}{|\xi|}, \nabla \psi(x)}  \,\rho(\xi) \,d \xi \, dx\,.
\end{equation}
A crucial observation from the theory of singular integrals, however, is that the bound $\|Du\|_{L^p}$ on the right-hand side of~\eqref{eq:DM1} is \emph{far too restrictive} when $\Ksf$ is not injective.\footnote{The situation worsens for $p=1$, where the requirement $Du \in L^1$ is too restrictive, even for elliptic operators (cf. \cite{KK_16,Ornstein_62}).} Instead, one would expect full control of the left-hand side in terms of  $\|D^\Ksf u\|_{L^p}$. The difficulty of this task is to circumvent the use of slicing arguments, which are unavailable for general operators (see~\cite{arroyo2020slicing}). From a technical standpoint, establishing the sharp control by $\|D^\Ksf u\|_{L^p}$ is one of the questions that motivated this work (see point (ii) in the next section).

\subsection*{Summary of the main results}The following is a brief summary of the main results established in this work.
\begin{enumerate}[(i)]\setlength{\itemsep}{5pt}
\item \textbf{Kernel representation.} In Definitions~\ref{def:spherical} and~\ref{def:radial} we introduce the \emph{spherical nonlocal operator} $\mathscr A_s$ (where $s$ denotes the radius) and the \emph{radial nonlocal operators} $\Afrak_\rho$ (where $\rho$ denotes a radial weight function) in terms of the difference quotient 
\[
	\Abb\left(\frac\xi{|\xi|}\right)\left[\frac{u(x + \xi) - u(x)}{|\xi|}\right]\,.
\] 
In the language of Du and Mengesha, we show that $\Abb$ is the ``canonical'' (linear) third-order tensor associated with $\Acal$ (compare this with Eqn.~\ref{eq:DM}). The main advantage of this perspective is that the operator defines the kernel and not vice versa. Therefore, avoiding a case-by-case analysis.

\item \textbf{Quantitative estimates.} Another novelty is that we are able to define the nonlocal operators on admissible distributions $u \in \mathscr D'$ (this includes $u \in L^p, \Mcal_b$ or even $u \in L_\loc^1,\mathscr D'$ under similar boundedness assumptions for $\Acal u$).\footnote{Here, $\Mcal_b$ denotes the space of finite  Borel measures.} To achieve this, we establish a new \emph{quantitative} control of the nonlocal operators in terms of the local operator. Namely, we show that 
\begin{align*}
	\|\mathscr A_s u\|_p & \le \|\Acal u\|_{p}\,, \\
	 \|\Afrak_\rho u\|_{p} & \le \|\rho\|_{L^1} \|\Acal u\|_{p}
\end{align*}
for an extended $L^p$-norm $\|\frarg\|_p$ on distributions. A new insight that stems from our techniques are the \emph{improved} estimates
\begin{align*}
	\|\mathscr A_s u\|_{L^1} &\le \|\Acal u\|_{\Mcal_b}\,, \\
 	\|\Afrak_\rho u\|_{L^1} &\le \|\rho\|_{L^1} \|\Acal u\|_{\Mcal_b}\,.
\end{align*}
These results, contained in Theorems~\ref{thm:spherical_localization} and~\ref{thm:Lp_bounds}, substantially improve upon the results contained in~\cite[Proposition 2.1]{DM_16}.

\item \textbf{Functional equivalence.} We establish sufficient and necessary conditions (see Propositions~\ref{prop:spherical_kernel} and~\ref{prop:kernel}) for the null spaces of $\Acal$ and $\Afrak_\rho$ to coincide among a range of $L^p$ spaces, that is, we discuss conditions on $s$ and $\rho$ under which 
\begin{align*}
	\ker_{L^p} \mathfrak A_\rho = \ker_{L^p} \ker \mathscr A_s \quad \text{and} \quad \ker_{L^p} \mathfrak A_\rho =  \ker_{L^p} \Acal\,.
\end{align*}

\item \textbf{Localization of the nonlocal operators.} We prove that $\Acal$ can be recovered from the non-local operators by means of a small-scale localization principle analogue to~\eqref{D}.  
More precisely, we show that if $u$ is a distribution with $\Acal u \in X = \{\mathscr D,\mathscr D',L^p,\Mcal_\text{area}\}$, then
\begin{align*}
	\mathscr A_s u  & \stackrel{X}\longrightarrow \Acal u  \qquad &\text{as $s \to 0$}\\
	 \Afrak_{\rho_\eps} u & \stackrel{X}\longrightarrow \Acal u \qquad &\text{as $\eps \to 0$}
\end{align*}
Here, $1 \le p < \infty$ and $\Mcal_\text{area}$ is the space of finite Borel measures endowed with the topology generated by the area-convergence of measures. 
In this regard, our localizations substantially refine the \emph{weak convergence} localizations established in~\cite[Theorem 3.12]{DM_16}, and generalize all the results established for gradients in~\cite{MS}. 
 These results are contained in Theorems~\ref{thm:spherical_localization},~\ref{thm:Lp} and~\ref{thm:M}.

 \item \textbf{Variational equivalence.} Motivated by the recent work of Cueto, Kreisbeck and Sch\"onberger~\cite[Remark~4.6]{Kreis2} (see also~\cite{K1}), we   observe that (see Propositions~\ref{thm:potential} and~\ref{thm:potential2}) the functional definitions of \emph{quasiconvexity} associated with $\mathscr A_s, \Afrak_\rho$, and $\Acal$ are identical.
 
 \item \textbf{Applications to fine properties of functions.} We give a general version of the Gauss-Green theorem, which, for continuous maps, conveys that $\Acal u \in \Mcal$ vanishes on $\Hcal^{n-1}$ sigma-finite sets. More precisely, 
 \[
 	u \in C(\R^n) \quad \text{and} \quad \Hcal^{n-1}(U) < \infty \quad \Longrightarrow \quad |\Acal u|(U) = 0. 
 \]
In Corollary~\ref{cor:fine} we generalize this to all \emph{Lebesgue continuous} functions. In fact, we show something slightly stronger (cf. Eqn.~\ref{eq:dos} and Theorem~\ref{thm:fine}). 
\end{enumerate}~\vskip1pt

\subsection*{Further comments} 
Our methods \emph{differ significantly} from the ones developed in previous works, as they are based on a generalization of the Gauss--Green theorem (see Section~\ref{a:gg}). An advantage of this is that our proofs can be easily adapted to find localizations for kernels associated with other shapes (such as cubes, regular polygons, or rings) upon a simple modification of the kernels. This even applies to non-symmetric profiles (although, in this case, the localization would yield approximations for certain non-isotropic norms). 

Of course, all of our results also apply to radial weights of the form
\[
	\rho = w I_{1 - s}\,, \qquad s \in (0,1)\,,
\]
where $I_\alpha$ is the $\alpha$-Riesz potential and $w$ is a radial cut-off function. Thus allowing one to work with (cf.~\cite[Section 2.5]{Kreis2}) non-local variants of an \emph{$s$-fractional} $\Acal$-gradient.

\section{Main definitions and results}\label{sec:2}

As usual, we write $\mathscr D(\R^n;V) = C^\infty_c(\R^n;V)$ to denote the space of smooth $V$-valued functions with compact support on $\R^n$. Accordingly, we write $\mathscr D'(\R^n;V)$ to denote its continuous dual, the space of $V$-valued distributions. In all that follows $\Mcal$ denotes a space of locally bounded Borel measures and $\Mcal_b$ its subspace of finite measures. Our main object of study will be the following nonlocal spherical and radial operators. We first define these for smooth and compactly supported maps:

\begin{definition}[Nonlocal spherical operator]\label{def:spherical}For a given test function map $u\in \mathscr D(\R^n;V)$ and $s >0$, we define the \emph{$s$-spherical non-local operator} 
\begin{align*}
	\mathscr A_s u(x) & \coloneqq n \fint_{\partial B_s} \Omega_\Abb(\omega) [\delta_\omega u(x)] \, dS(\omega) \\[1em]
	& \phantom{:}= n \fint_{\partial B_s} \Abb\left(\frac{\omega}{s}\right) \left[\frac{u(x+ \omega) - u(x)}{s}\right] \, dS(\omega), \qquad x \in \R^n.
\end{align*}
Here, $S$ denotes the uniform $(n-1)$-dimensional Hausdorff measure on $\partial B_s$ and $\Omega_\Abb$ is the zero-homogeneous profile of $\Abb$.
\end{definition}

Hereinafter $\rho : \R^n \to [0,\infty]$ will denote an integrable radial Borel function with profile $\hat \rho : \R \to [0,\infty]$. We define a nonlocal operator associated with $\Acal$ through the radial map $\rho$ as follows:

\begin{definition}[Nonlocal radial operator]\label{def:radial}
 For a given test vector field $u \in\mathscr D(\R^n;V)$, we define the \emph{nonlocal operator}
	\begin{align*}\label{eq:rho}
		\Afrak_\rho & u(x) \coloneqq \lim_{\delta \to 0^+} \left( n  \int_{\R^n \setminus B_\delta}  \Omega_\Abb\left(\frac{\xi}{|\xi|}\right)[\delta_{\xi} u(x)] \, \rho(\xi)\, d\xi \right)\\[1em]
	& \phantom{:}= \lim_{\delta \to 0^+} \, \left( n  \int_{\R^n \setminus B_\delta}  \frac{\Abb(h)[u(x +\xi) - u(x)]}{|\xi|^2} \, \, \rho(\xi)\, d\xi \right), \quad x \in \R^n.
	\end{align*}
\end{definition}
It is straightforward to verify that both $\mathscr A_s$ and $\Afrak_\rho$ are bounded integral  operators from $\mathscr D(\R^n;V)$ into $\mathscr D(\R^n;W)$ ---this will be formally discussed in {Sections}~\ref{sec:spherical} and~\ref{sec:radial}--- and  are also stable under integration by parts: 
\begin{proposition}[Integration by parts]
Let $u \in \mathscr D(\R^n;V)$ and $\varphi \in \mathscr D(\R^n,W)$ be test functions. Then
\[
	\int \mathscr A_s u \cdot \varphi = \int u \cdot \mathscr A_s^* \varphi
\]
and
 \[
 	\int \Afrak_\rho u \cdot \varphi = \int u \cdot \Afrak^*_\rho \varphi
 \] 
 Here $\mathscr A_s^*$ and
  $\Afrak^*_\rho$ are the respective spherical and radial operators associated with $\Acal^* = - \sum_{i = 1}^n A^t_i \partial_i$, the formal $L^2$-adjoint of $\Acal$. 
\end{proposition}

Duality allows us to \emph{uniquely and continuously} extend  $\mathscr A_s$ to all $V$-valued distributions:

\begin{definition}[On distributions] Let $T \in \mathscr D'(\R^n;V)$. The distributional spherical $\Acal$-gradient is defined as the distribution acting on test functions 
\[
	\mathscr A_s T(\psi) \coloneqq T(\mathscr A_s^* \psi)\,, \qquad \psi \in \mathscr D(\R^n;W)\,.
\]
\end{definition}
We shall see (Section~\ref{sec:spherical}) that the right-hand side above is well defined because $\mathscr A_s  \varphi$ is given by a convolution with a compactly supported measure. In general, it is not possible to directly extend $\Afrak_\rho$ to all distributions by duality. This occurs because $\Afrak_\rho  \varphi$ may not be compactly supported, even if $\varphi$ is. However, 
an extension exists for a class of distributions we call admissible:
\begin{definition}[Admissible distributions]\label{def:admissible} We say that a distribution $T \in \mathscr D'(\R^n;V)$ is \emph{admissible} if either
\begin{enumerate}\setlength{\itemsep}{0.2cm}
\item[(a')] $\rho$ is compactly supported (with no further requirement on $T$)
\end{enumerate} 
or $T$ satisfies one of the following conditions:
\begin{enumerate}\setlength{\itemsep}{0.2cm}
\item[(a)] The support of  $T$ is compact,
\item[(b)] $T \in L^p(\R^n;V)$  for some $1 \le p \le \infty$ or $T \in \Mcal_b(\R^n;V)$,
\item[(c)] $\Acal T \in L^p(\R^n;W)$  for some $1 \le p \le \infty$ or $\Acal T \in \Mcal_b(\R^n;W)$.
\end{enumerate}
If either (a), (b) or (c) hold, we say that $T$ is \emph{strongly admissible}.
\end{definition}

The fact that we can extend the radial nonlocal operator by duality to the class of admissible distributions is recorded in the following result (cf. Definition~\ref{def:radial_distributions} and~Lemma~\ref{lem:definition_radial}):

\begin{lemma}[The radial operator for distributions] The operator $\Afrak_\rho$
has a unique and continuous distributional extension (denoted the same way) over the class of admissible distributions: If $T$ is an admissible distribution, then 
\[
	\Afrak_\rho T(\psi) = T(\Afrak_\rho^* \psi) \quad \text{for all $\psi \in \mathscr D'(\R^n;W)$.}
\]
Moreover, the radial operator can be written as a distributional superposition of spherical ones as
\[
	\Afrak_\rho T = \int_0^\infty \omega_n \hat \rho(r) r^{n-1} \mathscr A_r T\, dr\,, \qquad \omega_n \coloneqq |B_1|\,,
\]
Here, the integral makes sense as a Bochner integral over a Banach space of distributions.
\end{lemma}

Given that we shall often work directly with distributions, it will be useful to identify locally integrable functions and measures with distributions. We also introduce the concept of extended $L^p$-norm:

\begin{definition}[Extended $L^p$ norms] Let $1 \le p \le  \infty$ and let $q$ be the the H\"older conjugate of $p$. We define the extended $L^p$-norm of a distribution $T\in \mathscr D'(\R^n;V)$ as	\begin{align*}
		\|T\|_{p} &  \coloneqq \sup \set{ T[\varphi]\,, }{\varphi \in \mathscr D(\R^n;V), \|\varphi\|_{L^q} \le 1}  
	\end{align*}
	\end{definition}
\begin{remark}If $1 < p \le \infty$, then $\|T\|_p$ is finite if and only if $T$ can be represented by integration by some $u \in L^p$. In particular $\| \frarg\|_p$ coincides with the $L^p$-norm on $L^p$ maps. Similarly, $\|T\|_1$ is finite if and only if $T$ can be represented by integration by some $\mu \in \Mcal_b$ and $\|T\|_1 = |\mu|(\R^n)$.
\end{remark}

It is straightforward to verify that, on locally integrable functions, the radial non-local operator is expressed by a principle value integral: 

\begin{remark}[Principal value]  Let $u  \in L^1_\loc(\R^n;V)$ be an admissible distribution. Then 
\[
	\Afrak_\rho u = \pv \, \left( n  \int  \frac{\Abb(\xi)[u(x +\xi) - u(x)]}{|\xi|^2} \, \, \rho(\xi)\, d\xi \right) 
\]
in the sense of distributions. Here, $u(x)$ is the precise Lebesgue representative of $u$ at $x$. We shall see (cf. {Lemma}~\ref{lem:integral_limit}) that if moreover $\|\Acal u\|_p < \infty$ for some $1 \le p \le \infty$, then $\Afrak_\rho u$ can be defined through an integral limit as in {Definition}~\ref{def:radial}.
\end{remark}

\subsection{Nonlocal spherical operators}Our first result establishes (extended) $L^p$-norm bounds and localizations for the spherical operators in terms of the $\Acal$-gradient. 

\begin{thmx}[Upper bounds and localization]\label{thm:spherical_localization}\ \begin{enumerate}
\setlength{\itemsep}{10pt}\setlength{\parskip}{10pt}
\item If $u \in \mathscr D(\R^n;V)$, then 
\[
	\mathscr A_s u \stackrel{\mathscr D}\longrightarrow \Acal u \qquad \text{as $s \to 0^+$}.
\]
and
\[
\int g(\mathscr A_s u) \le \int g(\Acal u)
\]
for all nonnegative convex functions $g : W \to [0,\infty)$. 

\item If $T \in \mathscr D'(\R^n;V)$, then
\[
	\mathscr A_s T \stackrel{\mathscr D'}\longrightarrow \Acal T \qquad \text{as $s \to 0^+$}.
\]
Moreover,
\[
	\|\mathscr A_s T\|_p \le \|\Acal T\|_p \quad \text{for all $1 \le p \le \infty$\,.}
\]

\item If $\Acal T \in L^p(\R^n;W)$ for some $1 \le p \le \infty$, then
\[
\mathscr A_s T \in (C \cap L^p)(\R^n;W).
\]
In moreover $1 \le p < \infty$, then the distributional convergence improves to 
\[
	\mathscr A_s T \stackrel{L^p}\longrightarrow \Acal T \qquad \text{as $s \to 0^+$}.
\]
\item If $\Acal T\in \Mcal_b(\R^n;W)$, then 
\[
\mathscr A_s T \in (L^\infty_\loc \cap L^1)(\R^n;W) 
\]
and the absolutely continuous measures $\{\mathscr A_s T \, \mathscr L^n\}_s$ converge to $\Acal u$ in the area sense of measures as $s \to 0$.  That is equivalent to requiring that
\[
	\lim_{s \to 0^+} \int f(\mathscr A_s T) =  \int f(\Acal^a T) + \int f^\infty(\Acal^s T) 
\]
for all continuous integrands $f: W \to \R$ with with a well-defined recession limit function $f^\infty : W \to \R$. Here,
\[
	f^\infty(w) \coloneqq \lim_{t \to \infty} \frac {f(tw)}{t}, \qquad w \in W,
\]
is the recession function of $f$ and
\[
	\Acal T = \Acal^a T\, \mathscr L^n +  \Acal^s  T, \qquad |\Acal^s u| \perp \mathscr L^n,
\]
is the Lebesgue--Radon--Nykod\'ym decomposition of $\Acal T$.
\end{enumerate}
\end{thmx}

The proof of this result is contained in {Section}~\ref{sec:spherical}. Thanks to a simple measure-theoretic lemma discussed in the Appendix (cf. Lemma~\ref{lem:Leb}), we obtain the following representation for the spherical operators (cf. Eqn.~\ref{eq:rep}).
\begin{corollary}\label{lem:GG}
	Let $T \in \mathscr D'(\R^n;V)$ and assume that $|\Acal T|$ is a locally bounded measure. Then,  
	\[
		\mathscr A_s T(x) =  \frac{\mathcal AT(B_s(x))}{\omega_n s^n}.
	\]
for almost every $x \in \R^n$.
\end{corollary}

\begin{remark}[On the continuity of $\mathscr A_s u$] In general, $\mathscr A_s u$ needs not to be continuous, even if $u \in L^1_\loc(\R^n;V)$ with $\Acal u \in \Mcal(\R^n;W)$. Take, for instance, the vector field on the plane given by
\[
	u(x_1,x_2) = \frac1{2\pi} \left(\frac{(x_1,x_2 + 1)}{|(x_1,x_2+1)|^2} - \frac{(x_1+1,x_2)}{|(x_1+1,x_2)|^2}\right)\,.
\]
It is straightforward to verify that $\diverg u = \delta_{(0,1)} - \delta_{(1,0)}$. In particular, the map $x \mapsto \diverg u(B_1(x))$ is not continuous at $0 \in \R^2$. \end{remark}

\subsection{Nonlocal radial operators} The localization result for the spherical operators builds directly into a swift proof of the estimates and localization results for the radial nonlocal operators. 

\begin{thmx}\label{thm:Lp_bounds}  
Let $T\in \mathscr D'(\R^n;V)$ be an admissible distribution. 
\begin{enumerate}[(1)]\setlength{\itemsep}{5pt}
\item The extended $L^p$-norm estimate
\[
	\|\Afrak_\rho T \|_p  \le  \|\rho\|_{L^1} \|\Acal T \|_p
\]
holds for all $1 \le p \le \infty$.
\item If $\Acal T$ is a measure, then
\[
	\int g(\Afrak_\rho T) \,d x \le \int g(\|\rho\|_{L^1}\Acal^a T) \,d x +  \|\rho\|_{L^1}\ \int g^\infty(d\Acal^s T)\,
\]
for all non-negative convex functions $g : W \to [0,\infty)$ with linear growth at infinity.
\item The extended $L^1$-estimate improves to
\[
\|\Afrak_\rho T \|_{L^1}  \le  \|\rho\|_{L^1} |\Acal T |(\R^n)\,
\]
provided that $\Acal T \in \Mcal_b$.
\end{enumerate} 
\end{thmx}

The proof of this result is a direct consequence of the superposition principle and the estimates for spherical operators (cf. Corollary~\ref{cor:superposition}), Proposition~\ref{prop:12}, and the discussion above Eqn.~\ref{eq:L1_lift}.

For locally summable maps with suitably bounded  $\Acal$-gradients, the operator $\Afrak_\rho$ is expressed, as one would expect, by a limit integral:
\begin{lemma}[Integral limit representation]\label{lem:integral_limit}
Let $u \in L^1_\loc(\R^n;V)$ be such that 
\[
\|\Acal u\|_p < \infty \qquad \text {for some $1 \le p < \infty$}.
\]  
Then, $\mathfrak A_\rho u$ is $p$-integrable and
\[
	g_\delta \longrightarrow \Afrak_\rho u \quad \text{in $L^p$,}
\]
where
\[
	g_\delta(x) \coloneqq  n \int_{\R^n \setminus B_\delta} \frac{\Abb(\xi)[u(x+\xi) - u(x)]}{|\xi|^2} \, \rho(\xi) \, d\xi.
\]
In particular,
\[
\Afrak_\rho u(x) = \lim_{\delta \to 0^+} \left(n \int_{\R^n \setminus B_\delta}  \frac{\Abb(\xi)[u(x+\xi) - u(x)]}{|\xi|^2} \, \rho(\xi) \, d\xi\right)
\]
for a.e. $x \in \R^n$.
\end{lemma}

 \begin{corollary}
Let $1 \le p < \infty$ and let $u \in L^1_\loc(\R^n;V)$ be such that $\|\Acal u\|_p < \infty$. The extended limit
	\[
		   \lim_{\delta \to 0^+}  \, \int \bigg| \int_{\R^n \setminus B_\delta}\frac{\Abb(\xi)[u(x+\xi) - u(x)]}{|\xi|^2} \, \rho(\xi) \dd \xi\,\bigg|^p dx\in [0,\infty] 
	\]
exists and coincides with $\|\Afrak_\rho u\|_p^p$.
 \end{corollary}
 
These two results follow from of Lemma~\ref{lem:11} and its proof.\\

\subsection{Localization} Here and in all that follows $\{\rho_\eps\}_{\eps > 0} \subset L^1(\R^n)$ is a family of radial functions satisfying~\eqref{eq:eps1}-\eqref{eq:eps}. Our first main localization result concerns a localization for test functions (cf. Corollary~\ref{cor:localization_distributions}):

\begin{lemma}\label{lem:Lp_extended}
Let $1 \le p \le \infty$ and let $\varphi \in \mathscr D(\R^n;V)$. 
Then
\[
	\Afrak_{\rho_\eps} \varphi \stackrel{L^p}\longrightarrow \Acal \varphi \qquad \text{as $\varepsilon \to 0^+$.}
\]
If moreover $\{\rho_\eps\}_\eps$ is uniformly compactly supported, then also
\[
 \Afrak_{\rho_\eps} \varphi \stackrel{\mathscr D}\longrightarrow \Acal \varphi   
 \qquad \text{as $\varepsilon \to 0^+$.}
\]

\end{lemma}

Next, we state the localization for strongly admissible distributions, together with an \emph{energy criterion}  for the extended $L^p$-norms of the distribution $\Acal u$. The precise statement, which follows by gathering the results contained in Corollaries~\ref{cor:localization_distributions} and~\ref{cor:limit_exists}, is the following:

\begin{corollary}
Let $T \in \mathscr D'(\R^n;V)$ be a strongly admissible distribution (or assume that the family $\{\rho_\eps\}$ is uniformly compactly supported). Then
\[
\Afrak_{\rho_\eps} T \stackrel{\mathscr D'}\longrightarrow \Acal T \qquad \text{as $\varepsilon \to 0^+$.}
\]
Moreover, the extended limit 
\[
	\lim_{\eps \to 0^+} \|\Afrak_{\rho_\eps} T\|_{p}  \in [0,\infty]
\]	
exists for all $1 \le p \le \infty$. 
\end{corollary}

This result and Corollary~\ref{cor:2} convey the following $L^p$-boundedness criterion: 

\begin{thmx}[$L^p$-localization]\label{thm:Lp}
Let $1 < p \le \infty$ and let $T \in \mathscr D'(\R^n;V)$ be a strongly admissible distribution (or assume that the family $\{\rho_\eps\}$ is uniformly compactly supported). Then, $\Acal T \in L^p(\R^n)$ if and only if 
\[
	\liminf_{\eps \to 0^+}   \|\Afrak_{\rho_\eps} T\|_{p}  < \infty.
\]

If either of the latter conditions hold and $1 < p < \infty$, then also
\[
	\mathfrak A_{\rho_\eps} T \longrightarrow \Acal T \quad \text{strongly in $L^p$.}
\]
\end{thmx}

\begin{remark}[Sharpness of the range of exponents] The first assertion still holds in the following sense: if  $\Acal u \in L^1$, then the  
representation of {Theorem}~\ref{thm:Lp} exists and agrees with the  $L^1$-norm of $\Acal u$ (cf. Corollary~\ref{cor:L1_sufficiency} below). In general, the converse fails due to concentration effects.  
On the other hand, the strong convergence $\mathfrak{A}_{\rho_\eps} T \to T$ in $L^\infty$ is straightforward for $T$ whenever the gradient $\Acal T$ is continuous and uniformly bounded. However, this convergence may fail for general $\Acal T \in L^\infty$ (see Example~\ref{example:no_infty}). 
In conclusion, the restriction $p \notin \{1,\infty\}$ on the exponent range in Theorem~\ref{thm:Lp} is sharp.\end{remark}

	 Our second main result covers the case  when $\Acal u$ is a vector-valued Radon measure. 
	Notice that, with our definition of extended norm, it holds that a $W$-valued distribution $T$ belongs to $\Mcal_b(\R^n;W)$, the space of $W$-valued Radon measures with finite total variation, if and only if $\|T\|_1 < \infty$, in which case it also holds $\|T\|_1 = |T|(\R^n)$. In particular, a direct consequence of Lemma~\ref{lem:Lp_extended} and Corollary~\ref{cor:2} is the following criterion and approximation result:

\begin{thmx}[Area localization]\label{thm:M}
Let $T \in \mathscr D'(\R^n;V)$ be a strongly admissible distribution (or else assume that the family $\{\rho_\eps\}$ is uniformly compactly supported). Then $\Acal T\in \Mcal_b(\R^n;W)$ if and only if 
\[
	\liminf_{\eps \to 0^+}   \|\Afrak_{\rho_\eps} T\|_1  < \infty.
\] 
Moreover, in this case  (cf. Theorem~\ref{thm:Lp_bounds}.3)
\[
	\mathfrak A_{\rho_\eps} T \, \mathscr L^n \longrightarrow \Acal T \quad \text{in the area sense of measures.}
\]
In particular, 
\[
	\lim_{\eps \to 0} \int f(\Afrak_{\rho_\eps} T) \, dx  = \int f(\Acal^a T) \, dx +\int f^\infty(\frac{\Acal T}{|\Acal T|}) \, d|\Acal^s T|
\]
for all continuous integrands $f: W \to \R$ with with a well-defined recession limit function $f^\infty : W \to \R$.\end{thmx}

The crucial advantage of the area convergence of measures (over the usual strict convergence) is that it discriminates the appearance of $L^1$ concentrations when the limit is in $L^1$ (see Lemma~\ref{lem:area_discriminates} in the Appendix). In particular, we obtain the following refinement of {Theorem}~\ref{thm:M}:

\begin{corollary}\label{cor:L1_sufficiency} 
Let $T \in \mathscr D'(\R^n;V)$ be such that $\Acal T \in L^1(\R^n;W)$. Then
\[
\mathfrak A_{\rho_\eps} T  \longrightarrow \Acal T  \quad \text{strongly in $L^1$.}
\]
\end{corollary}

\subsection{Functional equivalences} Let $s>0$ and let $u \in L^p(\R^n;V)$. A direct consequence of Theorem~\ref{thm:spherical_localization}.(2) is that
\[
 	\ker_{\mathscr D'} \Acal \subset \ker_{\mathscr D'} \mathscr A_s. 
 \]
On the other hand, Theorem~\ref{thm:Lp_bounds}.(1) conveys the set inclusion
 \[
 		\ker_{\mathscr D'} \Acal \subset \ker_{\mathscr D'} \mathfrak A_\rho\,. 
 \]
The reverse inclusion is a delicate question, which we will address next (at least partially).  To achieve this, let us recall that the Bessel function $J_\alpha$ of real order $\alpha > 0$ is defined on all $t \ge 0$ by its Poisson representation formula as
\begin{align*}
    J_\alpha(t) = \frac{t^\alpha}{2^\alpha \Gamma(\alpha +\frac 12)\Gamma(\frac 12)} \int_{-1}^{1} \cos(ts)(1-s^2)^\alpha \,\frac{ds}{\sqrt{1-s^2}}\,.
\end{align*}
Let us also recall the following explicit expression for the Fourier transform of the open unit ball and the unit sphere in terms of Bessel functions (see, e.g.,~\cite[p. 578]{Grafakos}): 
\[
	\widehat{1_{B_1}}(\xi) = \frac{J_{\frac n2}(2\pi|\xi|)}{|\xi|^{\frac n2}}\,, \quad \widehat{\sigma^{n-1}}(\xi) = c(n) \frac{J_{\frac{n-2}{2}}(2\pi|\xi|)}{|\xi|^\frac{n-2}{2}}\,, \qquad \xi \in \R^N\,.
\]
A basic property of Bessel functions is the following decay estimate
\[
	|J_\alpha(t)| \le C(\alpha) t^{-\frac12} \quad \text{for $t > 0$.}
\]
A consequence of these observations is the following (sharp) decay estimate for the characteristic function of the surface measure on the unit sphere:
\[
	|\widehat{\sigma^{n-1}}(\xi)| \le C(n) |\xi|^{-\frac{n- 1}{2}} \quad \text{for $\xi \in \R^n$}.
\]
In particular, this shows that, for $n \ge 2$, $\widehat{\sigma^{n-1}} \in L^p(\R^n)$ if and only if $p > \frac{2n}{n-1}$. For this and other facts concerning Bessel functions, we refer the reader to~\cite{grafakos2008classical} and references therein. 

Finally, let us introduce the set
 \[
    \mathscr V_\Acal \coloneqq \{\zeta \in \R^n : \rank \Abb(\zeta) > 0\}\,,
 \]
 of Fourier frequencies where the operator $\Acal$ is non-trivial. The next result addresses the equivalence of kernels for spherical operators on a range of exponents:

\begin{proposition}\label{prop:spherical_kernel}
	Let $s > 0$. 
	\begin{enumerate}
		\item If $1 \le p \le 2$, then      \begin{equation}\label{eq:kernels}
     \ker_{L^p(\R^n)} \mathscr A_s = \ker_{L^p(\R^n)} \Acal\,.
     \end{equation}
     		If moreover $n=1$, this equivalence holds for all $1 \le p < \infty$.

		\item If $p > \frac {2n}{n-1}$ (or $p = \infty$ when $n=1$) and $\Acal$ is non-trivial, 
		then for each $s>0$	there exists a radially symmetric function $u_s \in (C^\infty \cap L^p)(\R^n)$ satisfying
		\[	
			\mathscr A_s u_s = 0 \quad \text{and} \quad \Acal u_s \neq 0.
		\]
		In particular, for every $s > 0$,
		\[
			  \ker_{L^p(\R^n)} \Acal \subsetneq \ker_{L^p(\R^n)} \mathscr A_s\,.
		\]
		\item Let $\Tbb^n$ denote the $n$-dimensional flat torus and let $1 \le p \le \infty$. The kernel identity 
     \[
     \ker_{L^p(\Tbb^n)} \mathscr A_s = \ker_{L^p(\Tbb^n)} \Acal
     \]
    is equivalent with the constraint 
     \[
     J_\frac{n}2(2\pi s|m|) \neq 0\, \quad \text{for all $m \in \Zbb^n \cap \mathscr V_\Acal$\,.}
     \]
	\end{enumerate}
\end{proposition}
\begin{remark}[Particularities in dimension one]
    The fact that, in dimension $n=1$ the equality of kernels~\eqref{eq:kernels} holds for all $1 \le p < \infty$,  follows directly from the fundamental theorem of calculus (see the proof of Proposition~\ref{prop:spherical_kernel}). The equality fails for $p =\infty$ as every non-constant, $[0,1)$-periodic map, with zero average on $[0,1)$, belongs to $\ker_{L^\infty} \mathscr D_1$ (here, $\mathscr D_1$ is the $1$-spherical gradient operator).  More generally, the fact that the Bessel function has the simple explicit form
    \[
    J_{1/2}(t) = \sqrt{\frac{2}{t\pi}} {\sin t} \quad \text{for all $t \ge 0$}\,,
    \]
    implies its zeroes are precisely all the positive integers. In particular, the kernel of the $s$-scale spherical gradient $\mathscr D_s$ contains non-constant functions provided that  $s$ is a positive integer. 
    \end{remark}

We do not know if the range of the assertion in (1) can be lifted all the way up to $p \le \frac{2n}{n-1}$ (cf. (2)). For the gradient operator (or any other injectively elliptic operator for that matter), the question is related to finding the sharp range of exponents $p = p(n) \in [1,\infty]$ where the convolution operator $u \star 1_{B_1}$
is injective. 
We conjecture that it is precisely $p \in [1, \frac{2n}{n-1}]$:
\begin{conjecture}
        Let $n\ge 2$. The convolution operator 
    \[  
        L^p(\R^n) \to L^p(\R^n)	: u \mapsto u \star 1_{B_1} 
    \]
    is one-to-one for all $1 \le p \le \frac{2n}{n-1}$.
\end{conjecture}

One may wonder if the equivalence of kernels conveys the $L^p$-equivalence of the spherical operator and its local counterpart. Unfortunately, this is not the case because $J_{\frac n2}$ has infinitely many zeros and all $L^p$-multipliers must be uniformly bounded (see~\cite[Section 2.5.5]{Grafakos}). In particular, we get the following negative result:

\begin{corollary}[Failure of Korn's inequality]\label{cor:no_Korn} Let $1 \le p \le \infty$ and let $s > 0$. 
	There exists no constant $C > 0$ such that
	\[
		\|\Acal u\|_p \le C \|\mathscr A_s u\|_p 
	\]
	 for all $u \in \mathscr D(\R^n;V)$.
\end{corollary}

 To discuss the equivalence of kernels for radial operators, we will introduce the auxiliary probability function 
\begin{equation}\label{eq:multiplier}
    \mu_\rho(x) \coloneqq \int_0^\infty n \hat \rho(r) 1_{B_r}(x) \, \frac{dr}{r}\,, \qquad x \in \R^n\,.
\end{equation}
Casting the Bessel function representation of the Fourier transform of the unit ball into the Fourier transform of $\mu_\rho$ yields the following \emph{oscillatory integral} expression
\[
    \widehat{\mu_\rho}(\xi) =  \frac{1}{|\xi|^\frac n2}\int_0^\infty nr^{\frac{n}{2} -1} \hat \rho(r) J_{\frac n2}(2\pi r|\xi|) \,dr\,.
\]

This \emph{multiplier} is particularly important for the understanding of the nonlocal operators. Indeed, the discussion in Appendix~\ref{sect:multi} tells us that $\Afrak_\rho u = \mu_\rho \star \Acal u$ holds for all $u$ in the Schwartz space $\Scal(\R^n;V)$. In particular, 
\[
    \widehat{\Afrak_\rho u} = \widehat{\mu_\rho} \, \Abb \widehat u
\]
is the multiplier associated with $\Afrak_\rho$, in terms of $\Acal u$. The fact that $\mu_\rho$ is a probability function, reflects just another way to understand the $L^p$ boundedness of $\Afrak_\rho$. With these considerations in mind, it is thus natural to investigate when $\widehat{\mu_\rho}(\xi)$ is nonzero (pointwise) and therefore one-to-one as a multiplier for all $1 \le p \le 2$.
Let us begin by observing that, similarly to the spherical case, as a direct consequence of the Riemann--Lebesgue lemma, we get the following negative result (which follows  from~Corollary~\ref{cor:non_invertible}):

\begin{corollary}[Failure of Korn's inequality]\label{cor:no_Korn2} Let $1 \le p \le \infty$. 
	There exists no constant $C > 0$ such that 
	\[
		\|\Acal u\|_p \le C \|\Afrak_\rho u\|_p 
	\]
	 for all $u \in \mathscr D(\R^n;V)$.
\end{corollary}

This tells us that $\Afrak_\rho$ and $\Acal$ are not comparable in the $L^p$-norm. It is worth to mention that Poincar\'e inequalities hold in $L^p$ under additional assumptions of the weight function. For a full discussion on this topic, we refer the reader to~\cite[Section~4]{bellido}. In spite of the failure of a Korn-type inequality, we are able to show that their associated $L^p$-kernels coincide under suitable assumptions (see also~\cite[Lemma 2]{DuTian}  and the discussion below it):

 \begin{proposition}\label{prop:kernel} Let $\rho : \R^n \to [0,\infty]$ be a non-trivial radially symmetric summable function.
  \begin{enumerate}[(i)]  
     \item If $1 \le p \le 2$, then
     \[
     \ker_{L^p(\R^n)} \Afrak_\rho = \ker_{L^p(\R^n)} \Acal
     \]
      if and only if
     \[
        \widehat{\mu_\rho}(\xi)  \neq 0 \quad \text{for almost every $\xi \in \mathscr V_\Acal$\,.}
     \]
          \item Let $1 \le p \le \infty$, then
     \[
     \ker_{L^p(\Tbb^n)} \Afrak_\rho = \ker_{L^p(\Tbb^n)} \Acal
     \]
      if and only if 
     \[
      \widehat{\mu_\rho}(m)  \neq 0 \quad \text{for all $m \in \Zbb^n \cap \mathscr V_\Acal$\,.}
     \]
         \end{enumerate}
\end{proposition}
Notice that (i) or (ii) hold whenever $\widehat{\mu_\rho}$ is strictly positive. In general, positivity might be too restrictive to guarantee the injective character of the multiplier. However, this condition is still interesting given that there exist well-known mild conditions that ensure its positivity:

\begin{example} [{cf.~\cite[Theorem~5.1]{Cho}}] Let $\rho : \R^n \to [0,\infty]$ be a non-trivial radial function satisfying the following properties
   \begin{itemize}
    	\item $\rho$ is integrable,
        \item $\hat \rho$ is piecewise continuous,
        \item $r \mapsto r^\frac{n-3}{2} \hat \rho(r)$ is monotone decreasing and
        \[
        \lim_{r \to \infty} r^\frac{n-3}{2}\hat \rho(r) = 0.
        \]
   \end{itemize}
            Then, $\widehat \mu_\rho$ is strictly positive. In particular, $\widehat \mu_\rho$ is strictly positive for all radial functions of the form
  \[
  	\rho(x) =\frac{\chi_{B_1}(x)}{|x|^{n - s}}, \qquad s \in (0,1)\,.
  \]
  \end{example}
  
  \begin{remark}[On similar positivity conditions]
	The authors in~\cite{bellido} (see page 5, Example 2.1 and Section 4 in that source) have furnished similar (less general) conditions that ensure the positivity of $\widehat {\mu_\rho}$. The trade off in those conditions is that they are robust enough to ensure that the inverted multiplier $(\widehat {\mu_\rho})^{-1}$ is above a small radial power. This, in turn,  allow the authors to establish Poincar\'e inequalities for $u$ in terms of the nonlocal radial gradient $\mathfrak D_\rho$. It is worth to mention that there exist  profiles $\hat \rho$ violating this type of conditions and still satisfying the positivity of $\widehat{\mu_\rho}$. This is shown by the following example:
\end{remark}
  
 \begin{example}[Gaussian modifications]
    Let $\sigma>0$ and consider a normally distributed random variable with expected value $0$ and variance $\sigma$, that is, a Gaussian of the form
    \[
        G_\sigma(t) ={\frac {1}{\sigma {\sqrt {2\pi }}}}\exp \left(-{\frac {1}{2}}{\frac {|t|^{2}}{\sigma ^{2}}}\right)\,, \qquad t \in \R\,.
    \]
Now, define the radially symmetric probability function
    \[
        \rho_\sigma(x) \coloneqq |x|^2 G_\sigma(|x|)\,, \qquad x \in \R^n\,.
    \]
By the properties of Gaussians and exploiting the alternative expression from~\eqref{lem:Psi} in the Appendix, it follows that
\[
   \widehat{\mu_{\rho_\sigma}}(\xi) =c  G_{\frac 1\sigma}(|\xi|) > 0\,, \qquad  \xi \in \R^n\,.
\]	
This shows that $\rho_\sigma$ satisfies the algebraic conditions contained in Proposition~\ref{prop:kernel}.(i)-(ii).  In particular, 
\[
	\ker_{L^p(\R^n)} \Afrak_{\rho_\sigma} = \ker_{L^p(\R^n)} \Acal\,.
\]
Notice, however, that $\rho_\sigma$ is not above any radial power and therefore no Poincar\'e inequality as the ones established in~\cite{bellido} hold for these weight. 
\end{example}

\subsection{Invariance of quasiconvexity} Let $A = \{\Acal,\mathscr A_s,\Afrak_\rho\}$. We write
$\mathscr D_\sharp(\Tbb^n;X) \subset \mathscr D(\Tbb^n;X)$ to denote the space of smooth maps $u : \R^n \to X$ that are $Q\coloneqq [0,1)^n$-periodic and have zero mean, that is, 
\[
	\int_Q u(y) \, dy  = 0\,.
\]
Let us briefly recall the definitions of quasiconvexity introduced in~\cite{dacorogna2006weak} and further developed in~\cite{FM99}: 

\begin{definition}[Potential quasiconvexity] A continuous integrand $h : V \to\R$ is called \emph{$A$-quasiconvex} if 
\[
	\forall v \in V, \qquad h(v) \le \int_Q h(u(x) + v) \, dx 
\]
for all mean zero $u \in \mathscr D_\sharp(\Tbb^n;V)$ satisfying $Au = 0$.
\end{definition}
Naturally, this concept is tailored to understand the variational properties of sequences of the form $\{v_j\}_j \subset \ker A$. It can be see that, for all constant-rank operators and under standard $p$-growth of the integrand, the concept of potential $A$-quasiconvexity is equivalent with the lower semicontinuity of 
\[
	u \mapsto \int h(u(x)) \, dx, \qquad Au = 0,
\]
with respect to weak forms of convergence in $L^p$. The following equivalences are a direct consequence of the results discussed in the previous section:

\begin{proposition}\label{thm:potential} Let $h: V \to \R$ be continuous integrand and let $s > 0$ be a positive number satisfying 
\[
	J_\frac{n}2(2\pi s|m|) \neq 0 \quad \text{for all $m \in \Zbb^n \cap \Vcal_\Acal$.}
\]
Then $h$ is $\Acal$-quasiconvex if and only if $h$ is $\mathscr A_s$-quasiconvex. 
\end{proposition}

\begin{proposition}\label{thm:potential2}Let $h: V \to \R$ be continuous integrand and let $\rho : \R^n \to [0,\infty)$ be a nontrivial radially symmetric integrable function satisfying
\[
	\int_0^\infty r^{\frac{n-2}{2}} \hat \rho(r) J_{\frac n2}(2\pi r|m|) \,dr\, \neq 0 \quad \text{for all $m \in \Zbb^n \cap \Vcal_\Acal$.}
\]
Then $h$ is $\Acal$-quasiconvex if and only if $h$ is $\mathfrak A_\rho$-quasiconvex. 
\end{proposition}

\section{Spherical surface kernels}\label{sec:spherical} 
In this section, we discuss localizations concerning spherical kernels. As we shall see below, working with these lower-dimensional measure kernels arises naturally from the Gauss-Green theorem. The analysis of these kernels will be fundamental in establishing similar estimates and localizations for the radial kernels discussed in the introduction.   We shall from now on consider  the uniformly distributed spherical surface measure 
\[
m_s = \frac{ \Hcal^{n-1} \mres \partial B_s}{|B_s|}, \qquad s > 0.
\]
To it, we associate a \emph{surface} kernel $K_s :\R^n \to \mathrm{Hom}(V,W)$ by setting
\[
	K_s(d\omega) \coloneqq \Omega_\Abb(\omega) \, m_s(d\omega),
\]
where $\Omega_\Abb$ is the zero-homogeneous angular part of the symbol $\Abb$ map, i.e.,  
\[
	\Omega_\Abb(x) \coloneqq \Abb\left(\frac{x}{|x|}\right), \qquad x \in \R^n - \{0\}.
\]
For each $s>0$, $K_s$ defines a finite $\mathrm{Hom}(V,W)$-valued measure on $\R^n$ and its therefore a tensor-valued compactly supported distribution. When applied on a test map $\varphi \in \mathscr D(\R^n;V)$, the Gauss--Green theorem yields the identity
\begin{align*}
	K_s \star \varphi(x)  & =  \frac 1{\omega_n s^n} \int_{\partial B_s}  \Omega_\Abb(\omega)\varphi(x + \omega) \, d\mathcal S(\omega) \\
	&  = \frac{1}{\omega_ns^n} \int_{\partial B_s(x)} \Abb(\nu(\omega))\varphi(\omega) \, dS(\omega)\\
	& = \frac{\Acal \varphi(B_s(x))}{\omega_n s^n} = \frac{\chi_{B_s}}{|B_s|} \star \Acal \varphi .
\end{align*}  
On the other hand, the cancellation property 
\begin{equation}\label{eq:cancel}
	\int_{S^{n-1}} \Abb(\omega) \,dS(\omega) = 0 
\end{equation}
conveys the identity 
\begin{align*}
	\mathscr A_s \varphi &= n\fint_{\partial B_s}  \Omega_\Abb(\omega)[\delta_\omega \varphi(x)] \, dS(\omega)\\
	& = K_s \star \varphi = \frac{\chi_{B_s}}{|B_s|} \star \Acal \varphi\,.
\end{align*}
These identities have two important implications. The first one is that the operator $\mathscr A_s$ can be uniquely and continuously extended to the space of distributions $\mathscr D'(\R^n;V)$. In fact
\[
	\mathscr A_s T = K_s \star T, \qquad K_s \star T(\varphi) \coloneqq T( (\tilde K_s)^t \star \varphi)\,,
\]
where $(\tilde\frarg)$ denotes the reflection operation $x \mapsto -x$ so that $(\tilde K_s)^t(h) = \Omega_{\Abb^*} \, dm_s$ and 
\[
	\mathscr A_s T(\varphi) = T(\mathscr A^*_s \varphi)\,.
\]
Notice that a direct consequence of this identity is the integration by parts formula
\[
	\int \mathscr A_s \varphi \cdot \psi = \int (K_s \star \varphi) \cdot \psi = \int \varphi \cdot (\tilde K_s^t \star \psi) = \int \mathscr \varphi \cdot \mathscr A^*_s \psi
\]
for all test functions $\varphi \in \mathscr D(\R^n;V)$ and $\psi \in \mathscr D(\R^n;W)$. The second one is that  
\[
	\|\mathscr A_s \varphi\|_p \le \|\Acal \varphi\|_p \quad \text{for all $1 \le p \le \infty$.}
\]
Indeed, every probability measure is an $L^1$ Fourier multiplier with norm one (and hence also an $L^p$ multiplier for all $p$). A key point that will be used later is that the estimate on the right-hand side is independent of the parameter $s$. 
If moreover $\mathcal AT$ can be represented by a Radon measure, then  
$K_s \star T \in L^\infty_\loc(\R^n;W)$ since
\[
	\|K_s \star T\|_\infty \le \frac{ |\Acal T|(\bar B_r(x))}{\omega_n s^n} \,.
\]
More precisely, since $\mathscr A_s T \in L^1_\loc$, standard mollifier theory yields
\begin{equation*}
	\lim_{\eps \to 0^+}  (\mathscr A_s T)_\eps(x) = \mathscr A_sT(x) \quad \text{for a.e. $x \in \R^n$\,,}
\end{equation*}
where $(\frarg)_\eps$ denotes mollification with a standard mollifier at scale $\eps$. The associative property of convolutions and the fact that convolutions commute with distributional derivatives then gives
\begin{equation}\label{eq:cont}
\begin{split}
		\frac 1{|B_s|} \lim_{\eps \to 0^+} (\Acal  T)_\eps(B_s(x)) & =  
		\lim_{\eps \to 0^+}  (\mathscr A_s T)_\eps(x) \\
		&  = K_s \star T(x)\quad \text{for a.e. $x \in \R^n$\,.}
\end{split}
\end{equation}
Lemma~\ref{lem:Leb} in the Appendix implies that $|\Acal T|(\partial B_s(x)) = 0$ for almost every $x \in \R^n$ and therefore also the right-hand side above converges to $\Acal T(B_s(x))/|B_s|$ almost everywhere on $\R^n$ (see, e.g.,~\cite[Proposition~1.62]{APF}. In particular,
\begin{equation}\label{eq:rep}
	\mathcal AT \in \mathcal M \quad \Longrightarrow  \quad \mathscr A_s T \equiv \frac{\Acal T(B_s(\frarg))}{|B_s|}  \; \text{as maps in $L^\infty$}.
\end{equation}

\subsection*{Proof of Theorem~\ref{thm:spherical_localization}} From this discussion, we infer the following properties which validate the assertions contained in {Theorem}~\ref{thm:spherical_localization}:
\begin{enumerate}[(i)]
\item If $u \in \mathscr D(\R^n;V)$, then 
\[
	\mathscr A_s u = K_s \star u \stackrel{\mathscr D}\longrightarrow \Acal u \qquad \text{as $s \to 0^+$}.
\]
Indeed, by the properties of differentiation for convolutions, it suffices to observe the validity of the supremum norm estimate
\begin{align*}
	|\partial^\alpha (K_s \star u - \Acal u)(x)| & \le \fint_{B_s(x)} |\Acal \partial^\alpha u(y) - \Acal \partial^\alpha u(x)| \, dy \\
	& \le s\|D^ku\|_{L^\infty(B_s(x))}
\end{align*}
where $k = |\alpha| + 2$. This proves that $\mathscr A_s \varphi \to \Acal u$ in $C^\infty$. Since by definition 
\[
	\mathrm{spt}(\mathscr A_s \varphi) = \mathrm{spt}(\Acal \varphi \star 1_{B_s}) \subset \{x \in \R^n: \dist(x,\mathrm{spt}(\varphi)) \le s\}\,,
\]
we conclude that $\mathscr A_s \varphi \to \Acal u$ in $\mathscr D$.
\item If $T \in \mathscr D'(\R^n;V)$, then a by-product of the previous convergence is the distributional convergence
\[
	\mathscr A_s T \stackrel{\mathscr D'}\longrightarrow \Acal T \qquad \text{as $s \to 0^+$}.
\]
Indeed, observing that $-\Abb^t$ is the principal symbol associated with the formal adjoint $\Acal^*$ of $\Acal$, we deduce that
\[
	\mathscr A_s T = T (\tilde K_s^t \star  u) \longrightarrow T(\Acal^* u) = \Acal T(u).
\]
\item If $u \in \mathscr D(\R^n;V)$, then Jensen's inequality gives
\begin{equation}\label{eq:jensen1}
	F(K_s \star u) \le   F(\Acal u) \star \frac{\chi_{B_s}}{|B_s|} 
\end{equation}
for all convex maps $F: W \to [0,\infty]$. Then, Young's convolution inequality conveys the estimate 
\[
\|F(K_s \star u)\|_{L^1} \le \|F(\Acal u)\|_{L^1}.
\]
In particular, 
\[
	\|\mathscr A_s u\|_p = \|K_s \star u\|_{L^p} \le \|\Acal u\|_{L^p}  \quad \text{for all $1 \le p \le \infty$.}
\]
\item The lower semicontinuity of the extended $L^p$-norms under distributional convergence and  the previous estimate gives rise to the  extended estimate
\[
	\forall p \in [1,\infty], \quad \|K_s \star T\|_p \le \|\Acal T\|_p \quad \text{for all $T \in \mathscr D'(\R^n;V)$}.
\]
The key point is that the right-hand side does not depend on the parameter $s$.
In particular, again from lower semicontinuity and distributional convergence, we obtain
\[
	\limsup_{s \to 0^+} \|K_s \star T\|_p \le \|\Acal T\|_p \le \liminf_{s \to 0^+} \|K_s \star T\|_p\,.
\]
This means that the extended limit exists and, in fact,
\[
	\forall p \in [1,\infty], \quad \lim_{s \to 0^+} \|\mathscr A_s T\|_p = \|\mathcal A T\| \in [0,\infty]\,.
\]
\item If $\Acal T \in L^p$ for some $p \in [1,\infty]$, then~\eqref{eq:cont} and the fact that  
\[
|\Acal T|(\partial B_s(x)) = 0\quad \text{ for all $x \in \R^n$,}
\] 
imply 
\[
\mathscr A_s T \in C(\R^n;W) \quad \forall s > 0.
\]

\item For $1 < p < \infty$, the convergence of the $L^p$ norms implies that the distributional converge improves to the strong convergence
\[
	\mathscr A_s T  \stackrel{L^p}\longrightarrow \Acal T \quad \text{as $s \to 0^+$.}
\]
\item Lastly, if $\Acal T \in \Mcal$ and $g : W \to [0,\infty)$ is convex function with linear growth at infinity, then the expression~\eqref{eq:rep} leads to the estimate
\begin{align*}
	\int g (\mathscr A_s T(x)) \, dx & = \int g\left( \fint_{B_s} \,d \mathcal AT(y-x) \right) \, dx \\
	& \le \int \fint_{B_s} g(\mathcal A^a T(y-x)) \, dx  + \int \fint_{B_s} g^\infty(d\Acal^s T(y-x)) \\
	& \le \int g(\mathcal A^aT(x)) \, dx + \int g^\infty(d\Acal^s T(x)).
\end{align*}
Here, in passing to the first inequality we have made use of two facts: Firstly, the inequality $g(x + y) \le g(x) + g^\infty(y)$ for all $x,y\in \R^n$, which holds for all convex functions (see~\cite[Lemma~2.5]{KK_16}). And secondly, Jensen's inequality. The last inequality follows from Young's convolution inequality. From this, we deduce the upper bound
\[
	\limsup_{s \to 0^+} \int g(\mathscr A_s T) \, dx \le  I_g(\mathcal AT)
\] 
where 
\[
	I_g(\mu)  \coloneqq \int g(\mu^a) \, dx + \int g^\infty(\mu^s )\,.
\]
We are left to show the limit inferior has the same bound.  This, however, follows  from the fact that $\mathscr A_s T \toweakstar \mathcal AT$ in the weak sense of measures. Indeed, it is a standard fact the convexity of $g$ ensures that $I_g$ is sequentially lower semicontinuous with respect to the weal convergence of measures, and hence
\[
\liminf_{s \to 0^+} \int g(\mathscr A_s T) \, dx \ge I_g(\mathcal AT)\,.
\]
This shows that the extended limit
\[
\liminf_{s \to 0^+} \int g(\mathscr A_s T) \in [0,\infty]
\]
exists and coincides with $I_g(\mathcal AT)$. Of course, taking the area integrand $g(z) = \sqrt{1 + |z|^2}$, this is equivalent (cf.~\cite[Remark~4.2]{me}) to saying that  
\[
	\mathscr A_s T \, \mathscr L^n  \stackrel{\mathrm{area}}\longrightarrow \Acal T \quad \text{as $s \to 0^+$.}
\]
\end{enumerate}

\section{Radial kernels}\label{sec:radial}
Let $\rho(x) = \hat \rho(|x|)$ be a non-negative radial summable function and consider the kernel $\Kcal_\rho \in C^\infty(\R^n - \{0\};\mathrm{Hom}(V,W))$ defined by
\[
	\Kcal_\rho(z) \coloneqq n \, \Omega_\Abb\left(\frac{z}{|z|}\right)\frac{\rho(z)}{|z|}, \qquad z \in \R^n \setminus \{0\}.
\]
The mean-value zero property 
\[
	\int_{\partial B_s(0)} \Kcal_\rho(\omega) \, dS(\omega) = 0 \qquad \text{for all $s >0$} 
\]
ensures that the limiting convolution integral 
\begin{align*}
	\Afrak_\rho \varphi(x) 
	& \coloneqq \lim_{\delta \to 0} \int_{\R^n \setminus B_\delta(x)} \Kcal_\rho(y-x) [\varphi(y) - \varphi(x)] \, dy 
\end{align*}
is well-defined for all test maps $\varphi \in \mathscr D(\R^n;V)$ and all $x \in \R^n$, independently of the infinitesimal sequence $\delta \to 0^+$. Indeed, 
\begin{align*}
		\Afrak_\rho \varphi(x)  
		& = n \lim_{\delta \to 0}  \int \Omega_\Abb(h) \left[\frac{\varphi(x + h)}{|h|}\right] \, \rho(h) \, dy \\
		& = n \lim_{\delta \to 0}  \int_\delta^\infty \omega_n\hat \rho(r)r^{n-1} \left(n\fint_{\partial B_r} \Omega_\Abb(h) \left[\varphi(x + h)\right]\, dS(h) \right)\, dr\\
		& =  n \lim_{\delta \to 0}  \int_\delta^\infty \omega_n \hat \rho(r) r^{n-1}  \mathscr A_r\varphi(x) \, dr \\
		& = n \int_0^\infty \omega_n \hat \rho(r) r^{n-1}  \mathscr A_r\varphi(x) \, dr\,.
\end{align*}

\subsection*{The nonlocal operator on admissible distributions} As we have already discussed, when $\rho$ is not compactly supported, the nonlocal operator $\Afrak_\rho$ may not have an continuous extension to all $V$-valued distributions.

With these considerations in mind, we are now in the position to extend the radial nonlocal operator to all admissible distributions (cf. Definition~\ref{def:admissible}).  Recalling the convolution representation of spherical operators, we obtain the expression (cf. Eqn.~\ref{eq:multiplier})
\[
	\Afrak_\rho^* \psi  = \mu_\rho \star \Acal^* \psi =  \Acal^* (\mu_\rho \star \psi )\,, \qquad \psi \in \mathscr D(\R^n;W)\,.
\]
Notice, on the other hand, that the linear functional
\[
	\psi \mapsto   T(\Acal^* ( \psi \star \mu_\rho   ))
\]
is well defined and continuous on $\mathscr D(\R^n;W)$, provided that $T$ is an admissible distribution. In fact, this assignment is sequentially continuous on $T$ with respect to distributional convergence $T_j \stackrel{\mathscr D'}\to T$ provided that one of the admissibility conditions is satisfied along the sequence. Moreover, \[
T_\varphi (\Acal^* ( \psi \star \mu_\rho )) = T_{\Afrak_\rho \varphi} (\psi)
\]
for all $\varphi \in \mathscr D(\R^n;V)$. Here, $T_f$ denotes the distribution identified with a function $f$ through integration.
These two observations guarantee that the nonlocal operator $\Afrak_\rho$  
has a unique continuous extension (with respect to the convergence of distributions) for all admissible $T$. 
This justifies the following definition:

\begin{definition}[Radial operator on distributions]\label{def:radial_distributions} Let $T \in \mathscr D'(\R^n;V)$ be an admissible distribution. We define $\Afrak_\rho T$ as the distribution 
\[
	\psi \mapsto  T(\Afrak_\rho^* \psi)\,.
\] 
\end{definition}

We have the following analogue expression for the nonlocal radial operator on admissible distributions:
\begin{lemma}\label{lem:definition_radial} Let $\psi \in \mathscr D(\R^n;W)$ and let $T\in \mathscr D'(\R^n;V)$ be admissible in the sense of Definition~\ref{def:admissible}. Then, the map
	\[
		r \mapsto  \hat \rho(r) r^{n-1} \mathscr A_r T(\psi)\,, \qquad r > 0\,,
	\]
is measurable and absolutely integrable on $(0,\infty)$.
 
In particular,  the functional
\[
	\psi \mapsto \int_0^\infty \hat \rho(r) r^{n-1} \mathscr A_r T(\psi) \, dr\,
\]
defines a $W$-valued distribution and, in fact, 
\[
	R(T) \coloneqq n \int_0^\infty \hat \rho(r) r^{n-1} \mathscr A_r T(\psi) \, dr\,
\]
is well-defined as a Bochner integral over a Banach space of distributions. 
\end{lemma}
\begin{proof}
	First, observe that  $\mathscr A_r T(\psi)$ belongs to $\R$  and coincides with $T(\mathscr A_r^* \psi)$ for every $r>0$. Let us write $f(r) \coloneqq \hat \rho(r) r^{n-1} \mathscr A_r T(\psi)$. Being the composition of continuous functions, $r \mapsto \mathscr A_r T(\psi)$ is continuous on $(0,\infty)$. Since $\hat\rho(r)$ is measurable, it follows that $f$ is  also measurable in $(0,\infty)$.  Next, we prove that  $f$ is summable on $(0,\infty)$. We divide the proof into the different assumptions under the definition of admissibility (cf. Definition~\ref{def:admissible})

	If~(b) holds, then $|T(\mathscr A_r^* \psi)| \le \|T\|_p  \|\Abb\|_{L^\infty(S^{n-1})}  \|D\psi\|_{L^\infty}$ for all $r>0$ and hence the $|f|$ is summable on $(0,\infty)$. This also says that $\Afrak_\rho T$ is a first-order distribution. From this, it follows that $R(T)$ defines a Bochner integral:
	\begin{itemize}
	\item over $W^{-1,p}$ for $1 < p \le \infty$, 
	\item over $W^{-2,q}$, with $1 < q < \frac{n}{n-1}$, when $p = 1$. (Here, we are using that Morrey's theorem implies the compact embedding $\Mcal_b \cembed W^{-1,q}$.)
	\end{itemize}

If the support of $T$ is compact (assumption (a)), then there exists a positive integer $\ell$ (depending solely on $T$) such that 
\[
\forall r > 0, \quad |T(\mathscr A_r^* \psi)| \le   \ell \sup_{0 \le |\alpha| \le \ell}  \|D^\alpha \Acal^* \psi \|_{L^\infty(\overline {B_\ell})}\,,
\]
and we reach the same conclusion for $|f|$. 
Here, we have used that $\mathscr A_r^* \psi = \Acal^* \psi \star \frac{1_{B_r}}{|B_r|}$ and the fact that convolution with a probability measure does not increase the $L^\infty$-norm. In this case $R(T)$ is defined as a Bochner integral over a certain negative Sobolev space depending on $\ell$ and $p$.

If else $\rho$ is compactly supported (assumption (a')), then there exists $r_\psi > 0$ such that $T(\mathscr A_r \psi) = \chi T(\mathscr A_r \psi)$, where $\chi$ is a cut-off of the ball $B_{r_\psi}$. Therefore, the analysis of this case reduces to the previous one. Notice that, the space of first-order distributions over a compact set is naturally endowed with a norm that endows it with a Banach space structure. Therefore, $R(T)$ is well-defined as a Bochner integral over this space. This covers the case when either (a) or (a') are is satisfied.

If Assumption~(c) holds, then $\|\mathscr A_r T\|_p \le \|\Acal T\|_p$ for every $r> 0$ and it is hence it is straightforward to verify that $|f|$ is integrable. Similarly to the first case, $R(T)$ is then well defined as a Bochner integral over $L^p$ ($p > 1$) or $W^{-1,q}$  with $1 < q < \frac n{n-1}$ (for $p = 1$). 

This finishes the proof.
\end{proof}

A direct consequence of this lemma is that $\Afrak_\rho$ is expressed as a superposition of spherical operators on admissible distributions. This, in turn, conveys an immediate proof of the extended $L^p$ bounds for the radial operator on such distributions:

\begin{corollary}\label{cor:superposition} Let $T\in \mathscr D'(\R^n;V)$ be an admissible distribution. 
Then
\[
	\Afrak_\rho T = n \int_0^\infty \omega_n \hat \rho(r) r^{n-1} \mathscr A_r T \, dr
\]
in the sense of distributions and 
\[
	\|\Afrak_\rho T\|_p \le \|\rho\|_{L^1} \|\Acal T\|_p \qquad \text{for all $1 \le p\le \infty$.}
\]
\end{corollary}
\begin{proof}
Let $\psi \in \mathscr D(\R^n;W)$. Then, by the duality identity on spherical operators we get
\[
	R(T)(\psi) = \int_0^\infty n\omega_n \hat \rho(r) r^{n-1} T(\mathscr A_r^*\psi)  \, dr\,.
\]
Since $R(T)$ is a Bochner integral over a Banach space contained in the space of distributions, it follows that $T$ commutes with the integral and therefore
\[
	R(T)(\psi) = T\left( \int_0^\infty n \omega_n \hat \rho(r) r^{n-1} \mathscr A_r^*\psi   \, dr \right) = T(\Afrak_\rho^* \psi). 
\]
This demonstrates the distributional identity between $R(T)$ and $\Afrak_\rho T$. 

The extended $L^p$-bounds are then a direct consequence of the bounds spherical operators and H\"older's inequality  (which also holds for the extended norms):
\begin{align*}
	R(T)(\psi) & \le \int_0^\infty n\omega_n \hat \rho(r) r^{n-1} \|\mathscr A_r T\|_p \|\psi\|_{L^q}  \, dr\, \\
	& \le  \int_0^\infty n\omega_n \hat \rho(r) r^{n-1} \|\Acal T \|_p \|\psi\|_{L^q}  \, dr\, = \|\rho\|_{L^1} \|\Acal T\|_p \|\psi\|_{L^q}\,.
\end{align*} 
Taking the supremum over all $\psi$ with $\|\psi\|_{L^q} = 1$ yields the desired bound. 

The proof is complete.
\end{proof}

\begin{corollary}\label{cor:localization_distributions}
	Let $1 \le p \le \infty$ and let $\varphi \in \mathscr D(\R^n;V)$. Then
	\[
		\Afrak_{\rho_\eps} \varphi \stackrel{C^\infty}\longrightarrow \Acal \varphi \quad \text{and} \quad \Afrak_{\rho_\eps} \varphi \stackrel{L^p}\longrightarrow \Acal \varphi\, \quad \text{as $\eps \to 0^+$.}
	\]
The first convergence improves to convergence in $\mathscr D$ provided that  $\{\rho_\eps\}_\eps$ is uniformly compactly supported.

If $T \in \mathscr D'(\R^n;V)$ is a strongly admissible distribution (or if the family $\{\rho_\eps\}_\eps$ is uniformly compactly supported), then 
\[
	\Afrak_{\rho_\eps} T \stackrel{\mathscr D'}\longrightarrow \Acal T\, \quad \text{as $\eps \to 0^+$}.
\]
\end{corollary}
	\begin{proof}
		Let $p$ be a seminorm on $\mathscr D$ satisfying
		\begin{equation}\label{eq:seminorm}
		\begin{split}
			 p(\mathscr A_r \varphi - \mathcal A\varphi) & \longrightarrow 0 \quad \text{as $r \to 0^+$}.
		\end{split}
		\end{equation}
	We use that $\rho_\eps$ is a probability function, the convexity of $p$ (in the form of Jensen's inequality), and the second property of $p$ to deduce the bound
		\begin{align*}
			p(\Afrak_{\rho_\eps}\varphi - \Acal \varphi) & \le \int_0^\infty n\omega_n \hat \rho(r)r^{n-1} p(\mathscr A_r \varphi - \Acal \varphi) \, dr \\
			& \le \|\rho_\eps\|_{L^1(B_\delta)}\sup_{r \le \delta} \left\{ p(\mathscr A_r \varphi - \Acal \varphi)\right\} + 2\|\rho_\eps\|_{L^1(\R^n \setminus B_\delta)}p(\Acal \varphi)\,.
		\end{align*}
	By~\eqref{eq:seminorm} it follows that the first quantity on the right-hand side tends to zero as $\delta \to 0^+$. On the other hand, the properties of $\{\rho_\eps\}$ yield that, similarly, the second quantity tends to zero as $\delta \to 0^+$. From this, we get
	\[
	p(\Afrak_{\rho_\eps}\varphi - \Acal\varphi) \longrightarrow 0\,.
	\]
	Since the topologies in $C^\infty$ and $L^p$ are defined by  seminorms  satisfying the convergence assumption~\eqref{eq:seminorm}, the first assertion follows directly from the previous analysis. The improved convergence with respect to the topology in $\mathscr D'$ follows from the fact that  if  $\{\rho_\eps\}_\eps$ is uniformly compactly supported, then so it is $\{\Afrak_{\rho_\eps}\varphi\}_\eps$.
	
	The second assertion follows from the duality expression of $\Afrak_\rho$ in terms of spherical gradients 
	\[
		\Afrak_{\rho_\eps}T (\psi) = T(\Afrak_{\rho_{\eps}}^* \psi)
	\]
	and the fact that (cf. Corollary~\ref{cor:localization_distributions}) $\Afrak_{\rho_{\eps}}^* \psi \stackrel{\mathscr D}\to \Acal^* \psi$ whenever the family $\{\rho_\eps\}_{\eps >0}$ is uniformly compactly supported.
	\end{proof}
	
	\begin{corollary}\label{cor:limit_exists}
		Let $1 \le p \le \infty$. If $T \in \mathscr D'(\R^n;V)$ is an admissible distribution (or $T$ is a distribution and the family $\{\rho_\eps\}_\eps$ is uniformly compactly supported), then the extended limit 
		\[
			\lim_{\eps \to 0^+} \| \Afrak_{\rho_\eps} T\|_{p} \in [0,\infty]
		\]
exists and equals $\|\Acal T\|_p$.
	\end{corollary}
	
	\begin{proof}
		A direct consequence of Corollary~\ref{cor:superposition} is that
		\[
			\limsup_{\eps \to 0^+} \| \Afrak_{\rho_\eps} T\|_{p} \le  \| \Acal T\|_{p}\,.
		\]
		Since $\|\frarg\|_p$ is the supremum of bounded linear functionals, it follows that $\|\frarg\|$ is lower semicontinuous with respect to distributional convergence.  In particular, by the previous corollary, we get
		\[
			\|\Acal T\|_p \le \liminf_{\eps \to 0^+} \| \Afrak_{\rho_\eps} T\|_{p}\,.	\]
		Casting these two observations yields that limit exists and coincides with $\|\Acal T\|_p$.
	\end{proof}

When $\Acal T \in \Mcal_b$, we can say more. Namely, that
\begin{equation}\label{eq:11}
	\Afrak_\rho T \in L^1(\R^n,W)\,.
\end{equation}
To verify this claim, let us define
\[
	g_s(x) \coloneqq  \int_s^\infty n\omega_n\hat\rho(r)r^{n-1} \mathscr A_r T(x)\, dr.
\]
From~\eqref{eq:rep}, we deduce that 
\[
	|g_s(x)| \le  \frac{1}{\omega_n s^n} \|\rho\|_{L^1} \|\Acal T\|_1,
\]
which shows that each $g_s$ is a uniformly bounded measurable map. We shall discover next that the $\{g_s\}_{s >0}$ defines a Cauchy sequence in $L^1$ (as $s \to 0^+$). We may assume without any loss of generality that $s > t$ and apply Fubini's theorem to deduce the $L^1$-estimate
\begin{align*}
	\| g_t - g_s\|_{L^1} \, dx & \le  \int_t^s n\omega_n\hat\rho(r)r^{n-1} \|\mathscr A_rT\|_{L^1} \, dr \\
	 & \le  \|\rho\|_{L^1(B_s \setminus B_t)} \|\Acal T\|_{1} \, dr\,.
\end{align*}
The absolute continuity of the integral (over $\rho$) implies that $\{g_s\}_{s >0}$ is indeed a Cauchy sequence. By definition
\begin{align*}
	\Afrak_\rho T(\varphi) & = T(\Afrak_\rho^* \varphi)  \\
	& = \lim_{s \to 0^+} \int_s^\infty n\omega_n\hat\rho(r)r^{n-1} T(\mathscr A_r^* \varphi) \, dr \\ 
	& = \lim_{s \to 0^+} \int_s^\infty n\omega_n\hat\rho(r)r^{n-1} \mathscr A_r T (\varphi) \, dr\\ 
	& = \lim_{s \to 0^+} \int \dpr{g_s(x) ,\varphi(x)} \, dx.
\end{align*} 
From this, it follows that $g_s \mathscr L^n \stackrel{\mathscr D'}\to \Afrak_\rho T$. However, since we have shown that $g_s$ is convergent in $L^1$, we immediately deduce that (identifying the distribution with the $L^1$ map)
\begin{equation}\label{eq:L1_lift}
	\Acal T \in \Mcal_b \quad \Longrightarrow \quad \Afrak_\rho  u\in L^1\,.
\end{equation}

\begin{proposition}\label{prop:12}
Let $u \in \mathscr D'(\R^n;V)$ be such that $\Acal u \in \Mcal_b(\R^n;W)$. If $\rho \in L^1(\R^n)$ is a nonnegative radial probability density, then
\[
	\|g(\Afrak_\rho T)\|_{1} \le \|g(\Acal^a T)\|_{1} +\|g^\infty(\Acal^s T)\|_1
\]
for all lower convex functions $g : W \to [0,\infty)$ with linear growth at infinity. 
\end{proposition}
\begin{proof}From~\eqref{eq:11} it follows that 
  $g(\Afrak_\rho T) \in L^1_\loc$. Let $K$ be an arbitrary compact subset of $\R^n$. Appealing to Jensen's inequality and Young's convolution inequality, we deduce that
\begin{align*}
	\int_K g(\Afrak_\rho T(x)) \, dx & = \int_K  g\left(\int_0^\infty  K_s \star T(x) \, n\omega_n \hat \rho(s) s^{n-1} \, ds \right)  \, dx \\
	& \le \int_K \int_0^\infty g(K_s \star T) \, n\omega_n \hat \rho(s) s^{n-1} \, ds \, dx\\
	& = \int_0^\infty n\omega_n \hat \rho(s) s^{n-1} \,ds \int_K g(\Acal T \star \frac{\chi_{B_s}}{|B_s|}) \, dx.
\end{align*}
The inequality $g(a + b) \le g(a) + g^\infty(b)$ and subsequent applications of Jensen's and Young's inequality permit us to estimate the right-hand side above by
\begin{align*}
\int_K & g(\Acal^a T \star \frac{\chi_{B_s}}{|B_s|}) + g^\infty(\Acal^s T \star \frac{\chi_{B_s}}{|B_s|}) \, dx \\ 
& \le \int_K g(\Acal^a T) \star \frac{\chi_{B_s}}{|B_s|} \, dx + \int_K g^\infty(\Acal^s T) \star \frac{\chi_{B_s}}{|B_s|} \, dx \\
& \quad \le \|g(\Acal^a T)\|_1 + \|g^\infty(\Acal^s u)\|_1.
\end{align*}
This finishes the proof.
\end{proof}

\begin{corollary}\label{cor:2}
Let $T \in  \mathscr D'(\R^n;V)$.  
\begin{enumerate}
\item If $1 \le p  <  \infty$ and $\Acal T \in L^p$, then
\[
	\Afrak_{\rho_\eps}T \stackrel{L^p}\longrightarrow \Acal T \quad \text{as $\eps \to 0$}.
\]
\item If $\Acal T \in \Mcal(\R^n;W)$, then
\[
	\Afrak_{\rho_\eps} T \, \mathscr L^d \stackrel{area}\longrightarrow \Acal T \quad \text{as $\eps \to 0$}.
\]
\end{enumerate}
\end{corollary}

\begin{proof}
 First, we prove (1). From Theorem~\ref{thm:spherical_localization}.3 we already know that $\mathscr A_s u \to \Acal u$ in $L^p$ as $s \to^+ 0$. In particular, 
	\[
		\lim_{\delta \to ^+0} \omega(\delta) = 0, \qquad \omega(\delta) \coloneqq \sup_{s \le \delta} \|\Acal u - \mathscr A_su \|_{L^p}\,.
	\]
Now, since each $\rho_\eps$ is a probability measure, we can write
\begin{align*}
	 (\Afrak_{\rho_\eps} u - \Acal u)(x)  =  \int_0^\infty n\omega_n s^{n-1}\hat \rho_\eps(s) (\mathscr A_s u - \Acal u)(x) \, ds.
\end{align*}
Taking the $L^p$ norm of this quantity and appealing to Theorem~\ref{thm:spherical_localization}.2 we obtain, for a fixed $\delta>0$, the raw estimate
\begin{align*}
	\| \Afrak_{\rho_\eps} u - \Acal u\|_{L^p}  & \le  \int_0^\infty \omega_n s^{n-1}\hat \rho_\eps(s) \|\mathscr A_s u - \Acal u\|_{L^p} \, ds  \\
	& \le \omega(\delta) \|\rho_\eps\|_{L^1(B_\delta)} + 2\|\Acal u\|_{L^p} \|\rho_\eps\|_{L^1(\R^n \setminus B_\delta)}\,,
	\end{align*}
Thus, letting $\eps \to 0$ at both ends of the inequality and appealing to~\eqref{eq:eps} we deduce that
\[
	\limsup_{\eps \to 0} \| \Afrak_{\rho_\eps} u  - \Acal u\|_{L^p} \le \omega(\delta)\,.
\]
We can now let $\lim_{\delta \to 0} \omega(\delta) = 0$ to conclude that
\[
	\limsup_{\eps \to 0^+} \| \Afrak_{\rho_\eps} u  - \Acal u\|_{L^p} \le \lim_{\delta \to 0} \omega(\delta) = 0\,,
\]
as desired. This proves (1). 

For the proof of (2), it suffices to show that $\Afrak_{\rho_\eps} T \toweakstar \Acal T$ and (see~\cite[Section 2.1]{advances}) 
\[
	\lim_{\eps \to 0} \sqrt{1 + |\Afrak_{\rho_\eps}T|^2} \, dx \to \int \sqrt{1 + |\Acal^a T|^2} \, dx + |\Acal^s T|(\R^n)\,.
\]  
Let $g(v) = \sqrt{1 + v^2}$ so that $g^\infty(v) = |v|$. The first assertion follows from Corollaries~\ref{cor:superposition} and~\ref{cor:localization_distributions}, which imply that $\Afrak_{\rho_\eps} u$ converges to $\Acal u$ in the weak sense of measures to $\Acal u \in \Mcal_b$. The second assertion follows from a version of Reschetnyak's theorem~(see~\cite[Lemma~8]{adolfo_FA}) and by the previous proposition, which allows us to conclude that
\begin{align*}
	\limsup_{\eps \to 0} \|g(\Afrak_{\rho_\eps} u)\|_{L^1} & \le \|g(\Acal^a u)\|_{L^1} + |g^\infty(\Acal^s u)|(\R^n) \\
	& \le \liminf_{\eps \to 0} \|g(\Afrak_{\rho_\eps} u)\|_{L^1}.
\end{align*}
Thus, demonstrating that the limit $\lim_\eps \|g(\Afrak_{\rho_\eps} u)\|_{L^1}$ exists. We hence deduce that $\Afrak_{\rho_\eps} u$ converges in the area sense of measures to $\Acal u$. 

This finishes the proof. \end{proof}

In general, weak-$\star$ convergence in $L^\infty$ plus convergence of the $L^\infty$ norms do not guarantee strong convergence in $L^\infty$. This stems from the fact that $L^\infty$ is not a uniformly convex Banach space. One would therefore expect that the convergence in point (1) fails for $p =\infty$. 
Below we furnish a simple example that showcases this is precisely the case: 
\begin{example}[Failure of localization in $L^\infty$]\label{example:no_infty} 
	Consider the function  
	\[
		u(t) = |t|, \qquad t \in \R\,.
	\]
	Clearly, $u \in W^{1,\infty}_\loc$ and $u' \in L^\infty(\R)$ is given by the Heaviside function. A straightforward calculation yields the following explicit expression for its associated $s$-spherical gradient: for $s > 0$,
	\[
		\mathscr D_s u (t) = (u' \star \frac{1_{[-s,s]}}{2s})(t) = \begin{cases}
		 1 & \text{if $t \ge s$}\\
		 \frac{t}{s} & \text{if $t \in (-s,s)$}\\
		 -1 & \text{if $t \le -s$}
		\end{cases}
	\] 
For each $\eps>0$, consider the family of probability weights:
\[
	\rho_\eps(t) = \frac{1_{[\eps,2\eps]}(|t|)}{2\eps}\,, \qquad t \in \R\,.
\]
Notice that if $|t| \le \eps$, then the superposition formula for the nonlocal radial gradient gives (here we are using that $n =1$ and $\omega_1 = 2$)
\begin{align*}
	\Dfrak_{\rho_\eps} u(t) & = \frac 1\eps \int_\eps^{2\eps} \mathscr D_r u(t) \, dr \\
	& = \frac 1\eps \int_\eps^{2\eps} \frac tr \, dr = \frac t\eps \ln (2)\,.
\end{align*}
This shows that $\mathscr D_{\rho_\eps} u$ is linear on the interval $(-\eps,\eps)$. In particular, we are able to deduce the following uniform lower bound
\[
\|\mathfrak D_{\rho_\eps} u - u'\|_{L^\infty(-1,1)} \ge 1 - \ln (2) > 0\,.
\]
This comes to show that, for this particular choice of weights, the nonlocal operators $\mathfrak D_{\rho_\eps} u$ fail to converge to $u'$ in $L^\infty$ as $\eps \to 0$.
\end{example}

The next result shows that $\mathfrak A_\rho u$ can be expressed by the principal value integral holding for smooth maps whenever $u$ is locally summable and $\mathfrak A_\rho$ is $L^p$-bounded.

\begin{lemma}\label{lem:11}
Let $1 \le p < \infty$ and let $u \in  L^1_\loc(\R^n;V)$ be an admissible distribution. Then
\[
\|\Afrak_\rho u\|_p^p \le \liminf_{\delta \to 0} \int \left| \int_{\R^n \setminus B_\delta} K_\rho(h)[u(x+h) - u(x)]\, dh \right|^p. 
\] 
The limit above exists, commutes with the integral, and
\[
	\Afrak_\rho u(x) = \lim_{\delta \to 0^+} \int_{\R^n \setminus B_\delta} K_\rho(h)[u(x+h) - u(x)] \, dh
\]
for almost every $x \in \R^n$.
\end{lemma}

\begin{proof}
Let $\varphi \in \mathscr D(\R^n;V)$. Since the integrand 
\[
(x,h) \mapsto \dpr{\Kcal_\rho(h)u(x+h),\varphi(x) - \varphi(x+h)}
\]
is absolutely integrable on $\R^n \times \R^n$, we deduce from Fubini's theorem  that
\begin{align*}
	\dpr{u , \Afrak^*_\rho \varphi} 
	& =  \int\int\dpr{\Kcal_\rho(h)u(x+h),\varphi(x) - \varphi(x+h)} \, dx \, dh.
\end{align*}
Now, the right-hand side can be expressed as the limit 
\begin{align*} 
	& = \lim_{\delta \to 0} \int_{B_\delta^c}\int\dpr{\Kcal_\rho(h)[u(x+h) - u(x)],\varphi(x) } \, dx \, dh \\
		& =\lim_{\delta \to 0}  \int \int_{B_\delta^c} \dpr{\Kcal_\rho(h)[u(x+h) - u(x)],\varphi(x) } \, dh \, dx \\
			& =\lim_{\delta \to 0}  \int\dpr{ \int_{B_\delta^c} \Kcal_\rho(h)[u(x+h) - u(x)]\, dh ,\varphi(x) }  \, dx \\
	& \le \|\varphi\|_{L^q} \liminf_{\delta \to 0}   \left(\int \left| \int_{B_\delta^c} \Kcal_\rho(h)u(x+h)\, dh \right|^p dx \right)^\frac{1}{p}.
\end{align*}
Taking the supremum over all $\varphi \in \mathscr D(\R^n;W)$ with $\|\varphi\|_{L^q} = 1$, where $q$ is the H\"older conjugate exponent of $p$, we conclude that 
\[
	\|\Afrak_\rho u\|_p^p \le \liminf_{\delta \to 0}  \int \left| \int_{B_\delta^c} K_\rho(y-x)u(y) \, dy \right|^p.
\]
This proves the first assertion.

The previous chain of equalities also shows that, whenever $\|\Acal u\|_p < \infty$ for some $1 \le p \le \infty$, then
\[
	\Afrak_\rho u = \pv  \int \Kcal_\rho(h)[u(x+h)] \, dh 
\]
in the sense of distributions on $\R^n$.
Let us now assume that $|\Acal u|$ is a locally finite measure (which holds whenever $\|\Acal u\|_p < \infty$). Since for every $r > 0$ the function $x \mapsto |\Acal u|(B_r(x))$ is Borel measurable, it follows that the non-negative function defined by
\[
	  f(x) \coloneqq \int_0^\infty \hat \rho(r) r^{-1}|\Acal u|(B_{r}(x))\varphi(x) \, dr \in [0,\infty], \qquad x \in \R^n,
\]
is also Borel measurable on $\R^n$. Moreover,  
 Fubini's theorem and Young's convolution inequality convey that $f$ is $p$-integrable  with $\|f\|_{L^p} \le \|\rho\|_{L^1} \|\Acal u\|_{p}$. 
Let us now define the auxiliary measurable maps 
\[
	g_\delta(x) \coloneqq \int_{\R^n \setminus B_\delta} \Kcal_\rho(h)[u(x+h) - u(x)]\, dh, \quad x \in \R^n.
\]
Suppose that $t > s > 0$ are arbitrary. We have
\begin{align*}
	\int |g_{t} - g_s|^p \, dx & \le \int \int_{B_t \setminus B_s} |\Kcal_\rho(h) [u(x + h)]|^p \, dh \, dx  \\
	& \le \int \int_s^t n\omega_n\hat \rho(r)r^{n-1} |K_r \star u(x)|^p \, dr \, dx\\
	& \le \int_s^t \int  n\omega_n\hat \rho(r)r^{n-1} |K_r \star u(x)|^p \, dx \, dr\\
	&  \le  \|\rho\|_{L^1(B_t \setminus B_s)}  \|\Acal u\|_{p}^{p}. 
\end{align*}
From this and the absolute continuity of the integral on summable functions, it follows that $\{g_\delta\}_{\delta > 0}$ is a Cauchy sequence (as $\delta \to 0^+$) in $L^p(\R^n;W)$. Since also 
\[
g_\delta \longrightarrow \mathfrak A_{\rho} u \quad \text{as $\delta \to 0^+$ \; in the sense of distributions}.
\] 
we conclude that 
\[
	g_\delta \longrightarrow \mathfrak A_{\rho} u \quad \text{as $\delta \to 0^+$ \; in $L^p(\R^n;W)$,}
\]
This finishes the proof.
\end{proof}

\subsection*{Proof of Proposition~\ref{prop:spherical_kernel}} 
The lower bound
\[
	\|\Afrak_\rho u\|_p	\le \|\rho\|_{L^1} \|\Acal u\|_p
\] 
conveys the set contention $\ker_{\mathscr D'} \Acal \subset \ker_{\mathscr D'} \mathfrak A_\rho$. We are therefore only left to prove the reverse set inclusion. By standard regularization and linearity of the operators, it suffices to work with smooth maps only.

In the following we will use the fact that (see.~\cite[\S B.4]{grafakos2008classical}) 
\[
    \widehat{1_{B_1}}(\xi) = |\xi|^{-n/2} J_{n/2}(2\pi|\xi|)\,.
\]
Denoting $\chi_r \coloneqq \frac{1_{B_r}}{|B_r|}$, we get
\begin{align*} 
\widehat{\chi_r}(\xi) & = \Fcal(\frac{1_{B_1}(r^{-1} \frarg)}{\omega_n r^n})(\xi)  = \omega_n^{-1} \widehat{1_{B_1}}(r \xi) \\
& = \omega_n^{-1} |r\xi|^{-n/2} J_{n/2}(2\pi r|\xi|) \eqqcolon G_r(\xi)\,.    
\end{align*}

\emph{Proof of (1).} The assumption $1 \le p \le 2$ guarantees that $\widehat u$ is properly a function. In particular, it follows that the distributional Fourier transforms of $\Acal u$ and $\mathscr A_s u$ are given by (locally integrable) functions. Namely,  
\[
    \Fcal(\Acal u)(\xi) = \Abb(\xi) \widehat u(\xi) 
\]
and 
\[
\Fcal(\mathscr A_s u)(\xi) = \Fcal(\Acal u \star \chi_s)(\xi) = G_s(\xi) \Fcal(\Acal u)(\xi)\,.
\]

Recalling that the set of zeroes of $J_{n/2}$ is discrete and hence also of $G_s$, we deduce that $\Fcal(\mathscr A_s u) = 0$ if and only if $\Fcal(\Acal u) = 0$ as functions. Since Fourier inversion applies in the sense of tempered distributions, we conclude that $\mathcal A u = 0$ if and only if $\mathscr A_s u = 0$. 

The fact that the proof works for all $1 \le p < \infty$ in dimension one follows directly from the fundamental theorem of calculus and the fact that intervals tessellate the real line. Indeed, in one dimension we can reduce the problem to the case when $\Acal = c (d/dt)$ for some (possibly zero) constant $c$. Hence, either $c=0$ and $\Acal = \mathscr A_s$, or (without loss of generality) $c = 1$ and
\[
	\mathscr A_s u(t) = \int_t^{t + s} \Acal u  =  \int_t^{t + s} u' =  u(t +s) - u(t) 
\] 
for all $t \in \R$. From this, it follows that $\mathscr A_s u = 0$ if and only if $u$ is $s$-periodic. However, for $1 \le p < \infty$, the only $p$-integrable $s$-periodic function is the zero function. This shows that $\mathscr A_s u = 0$ if and only if $u = 0$ (or equivalently, that $\Acal u = 0$). This proves (1).

\emph{Proof of (2).} The proof that for $p = \infty$ and $n=1$, one can construct a counterexample follows directly from the proof of (1): every  continuous $s$-periodic function is in the kernel of $\mathscr A_s$. We may hence assume that $n \ge 2$. Let $\sigma_n$ denote the surface measure on $S^{n-1}$. It is well-known that (see, e.g., ~\cite[B.4]{grafakos2008classical})
\[
    \widehat{\sigma_n}(\xi) =2\pi |\xi|^{-\frac{n-2}2} J_{\frac{n-2}2}(2\pi|\xi|)\,.
\]
Since Bessel functions decay as $r^{-1/2}$ at large scales, it follows that $|\widehat{\sigma_n}(r\xi)| \le Cr^{-\frac{n-1}{2}}$ for large $r$.   
Since also $\widehat{\sigma_n}$ is radially symmetric, it then follows that $\widehat{\sigma_n} \in L^p(\R^n)$ for all $p > \frac{2n}{n-1}$. 

On the other hand, the assumption that $\Acal$ is non-trivial allows us to find a vector $v \in S^{n-1} \cap \Vcal_\Acal$ satisfying $|\Abb(\nu)v| \ge \delta > 0$ for some $\nu \in S^{n-1}$. Now, for a given $s > 0$, let $r = r(s) > 0$ be such that $2\pi rs$ is a zero of $J_\frac{n}2$ so that
\[
	J_{\frac n2}(2\pi s |\xi|) = 0 \quad \text{for all $\xi \in rS^{n-1}$\,.}
\]
Define (by distributional Fourier inversion) the function 
\[
\widehat{u_s}(x) \coloneqq v {d\sigma_n(r^{-1}|\xi|)}(x), \qquad x \in \R^n.
\]
Since $\sigma_n$ is compactly supported, it follows that $u_s \in (C^\infty \cap L^p)(\R^n)$ for all $p > \frac{2n}{n-1}$. Moreover, by construction,   
\begin{align*}
    \Fcal(\mathscr A_s u_s)(\xi) & = \Abb \Fcal(u_s \star \chi_s)(\xi) \\
    & = (\Abb(\xi)[v]) \frac{J_\frac{n}2(2\pi s|\xi|)}{\omega_n  |s \xi|^{n/2} }  \, d\sigma_n(r^{-1} |\xi|) = 0\,. 
\end{align*}
This shows that $u_s \in \ker_{L^p} \mathscr A_s$. 
On the other hand, by construction, and by continuity and homogeneity of $\Abb$, we get
\[
    |\Fcal(\Acal u_s)(\xi)| = |\Abb(\xi)[v]|  {d\sigma_n(r^{-1} |\xi|)} \ge \frac{r\delta}2{d\sigma_n(r^{-1} |\xi|)} \,.
\]
for all $\xi \in rS^{n-1}$ in a sufficiently small neighbourhood of $r\nu$. This shows that $u_s \notin \ker_{L^p} \Acal$, as desired.

\emph{Proof of (3).} As mentioned earlier, it suffices to prove the assertion for $u \in \ker \mathscr A_s \cap C^\infty(\Tbb^n)$. 
The Fourier series' coefficients of $\Acal u$ and $\mathscr A_s u$ are given by
\[
    \mathfrak F (\Acal u)(m) = \Abb(m) \Ffrak u(m)  \quad \text{for all $m \in \Zbb^n$}
\]
and (recall that $G_s$ is continuous)
\[
    \mathfrak F (\mathscr A_s u)(m) = \begin{cases}
     G_s(m) \Abb(m) \Ffrak u(m)  &  \text{for all $m \in \Zbb^n - \{0\}$}\\
    0 & \text{else}
    \end{cases}\,.
\]
Therefore, the condition $\mathfrak F(\mathscr A_s u) \equiv 0$ is equivalent to the condition
\begin{equation}\label{eq:torus_As}
   \forall m \in \Zbb - \{0\}, \qquad J_\frac{n}2(2\pi s|m|) \neq 0 \quad \Rightarrow \quad \Ffrak u(m) \in \ker \Abb(m) 
\end{equation}
For simplicity lets us denote by (a) and (b) the respective conditions in (3), which we want to prove are equivalent. To see that (a) implies (b), we argue by a contrapositive argument $\neg (b) \to \neg (a)$: assume we can find $m \in \Zbb^n - \{0\}$ be such that $\Abb(m)$ is non-trivial (or equivalently that $m \in \Vcal_\Acal$) and $J_{n/2}(2\pi s|m|) = 0$. In particular, we may find $v \in V$ such that $\Abb(m)v \neq 0$. Setting $u(x) \coloneqq v e^{2\pi \mathrm{i} x \cdot m}$, we discover that
\[
    \Acal u(x) = (2\pi \mathrm i e^{2\pi \mathrm i x \cdot m}) \Abb(m)v \neq 0
\]
and 
\[
    \mathscr A_s u(x) = |sm|^{-n/2}J_\frac{n}2(2\pi s|m|) \Acal u(x) = 0\,.
\]
This proves $\neg (a)$, as desired. The implication $(b) \to (a)$ follows directly from~\eqref{eq:torus_As} and the Fourier coefficients inversion formula.

This proves (3) and the proof is finished. \qed

\subsection*{Proof of Proposition~\ref{prop:kernel} } 

\emph{Proof of (i).} Let $u \in L^p(\R^n)$. As before, the assumption $1 \le p \le 2$ guarantees that the distributional Fourier transform of $u$ is a locally integrable map and hence $\Fcal(\Acal u)$ is given by a locally integrable map. In particular $\Acal u = 0$ if and only if
\[
    \Fcal(\Acal u)(\xi) = \Abb(\xi) \widehat u(\xi) = 0 
\]
for almost every $\xi \in \R^n$. Since $u$ is an admissible distribution, the distributional expression 
\[
\Afrak_\rho u =  \int_0^\infty n\omega_n r^{n-1} \hat \rho(r) (\Acal u \star \chi_r)  \, dr\,, \qquad \chi_r \coloneqq \frac{1_{B_r}}{|B_r|}.
\]
commutes with the Fourier transform (in the sense of distributions). This yields the following identity for the distributional Fourier transform of $\Afrak_\rho u$:
\begin{align*}
    \Fcal(\Afrak_\rho u) =  n \left(\int_0^\infty \omega_n r^{n-1} \hat \rho(r) \widehat{\chi_r} \, dr \right) \Abb \widehat u\,.
\end{align*}
Once more, we use that $\widehat u$ is a locally integrable map to deduce that $\Fcal(\Afrak_\rho u)$ is also a locally integrable given by
\begin{align*}
    \Fcal(\Afrak_\rho u)(\xi) = & \left(\int_0^\infty nr^{\frac{n-2}{2}} \hat \rho(r) J_{n/2}(2\pi r|\xi|) \,dr \right)  \frac{\Abb (\xi) \widehat u(\xi)}{|\xi|^{n/2}}\\
    = &   \widehat{\mu_\rho} \Abb(\xi) \widehat u(\xi)\,
\end{align*}
for almost every $\xi \in \R^n$. 
From this expression, the continuity of $\widehat{\mu_\rho}$, and the Fourier inversion for tempered distributions, it follows that the condition $\widehat{\mu_\rho}(\xi) \neq 0$ almost everywhere on $\R^n$ is equivalent to the identity $\ker_{L^p} \Acal = \ker_{L^p} \Afrak_\rho$ for all $p \in [1,2]$. This proves (i).

\emph{Proof of (ii).} The proof of this statement follows closely from the previous case and the proof of (3) in Proposition~\ref{prop:spherical_kernel}.

This finishes the proof.\qed

\section{Applications to fine properties}\label{a:gg}

Let $u$ be a locally integrable map. 
The coarea formula implies that the Fubini trace $u^*$ of $u$ exists on $\partial B_s(x)$ for almost every $s \in (0,\infty)$ and
\[
	\int_{B_r(x)} u = \int_0^r \int_{\partial B_s(x)} u^* \, dS.
\]
We have the following generalization of the Gauss--Green theorem on balls:
\begin{lemma}\label{lem:GG}
	Let $u \in L^1_\loc(\R^n;V)$ and  assume that $\Acal u \in \Mcal(\R^n;W)$. Let $x \in \R^n$ be given. For almost every $s \in (0,\infty)$, where the Fubini trace $u^* \in L^1(\partial B_s(x);V)$ of $u$  exists on $\partial B_s(x)$ and $|\Acal u|(\partial B_s(x)) = 0$, it holds
	\[
		\Acal u(B_s(x)) = - \int_{\partial B_s(x)} \Abb(\nu(\omega))[u^*(\omega)]  \, dS(\omega).
	\] 
	Here, $\nu$ is the exterior unit normal of $B_s(x)$. 
\end{lemma}
\begin{proof}
We may, without loss of generality assume that $x = 0$. We argue by approximation. Let $u_\eps = \omega_n\eps^{-n}\chi_{B_\eps}  \star u$. If $s \in (0,1)$ is such that $u^*$ exists on $\partial B_s(x)$, then it also holds
\begin{align*}
\int_{\partial B_s} |u^* - u^\eps| \, dS & \le \int_{\partial B_s} \fint_{B_\eps(y)} |u^*(y) - u(z)|\, dz\,  dS(y) \\
& \le \frac1{2\eps} \int_{s-\eps}^{s +\eps} \int_{\partial B_\rho} |u(sy/\rho) - u(z)|\, dS \, d\rho = \mathrm{O}(\eps)\,.
\end{align*}
Letting $\eps \to 0$ and applying the Gauss--Green theorem for smooth maps, we observe that if $|\Acal u|(\partial B_s(x)) = 0$, then
\[
	\int_{\partial B_s} \Abb(\nu)[u^*] \, dS \leftarrow \int_{\partial B_s} \Abb(\nu)[u_\eps] \, dS  = \Acal u_\eps(B_s(x)) \to \Acal u(B_s(x))\,.
\]
The finishes the proof.
\end{proof}
\begin{corollary}
Let $u \in L^1_\loc(\R^n;V)$. Then, 
\[
	\Acal u = 0 \quad \text{in the sense of distributions on $\R^n$}
\]
if and only if for every $x\in \R^n$  it holds 
\[
	 \int_{\partial B_s(x)} \Abb(\nu(h)) [u^*(h)] \,dS(h)  = 0,
\]
for every $s \in (0,\infty)$, where the trace $u^* \in L^1(\partial B_s(x);V)$ exists. 
\end{corollary}

We recall some of the basic concepts associated with maps' fine properties. A point $x \in \R^n$  is called a Lebesgue point of $u$ (or an approximate continuity point) if and only if 
\[
	\limsup_{r \to 0} \fint_{B_r(x)} |u(y) - z| = 0 \quad \text{for some $z \in V$}.
\] 
We write $S_u$ to denote the set of all (Lebesgue) discontinuity points of $u$, where the limsup above is strictly positive. 
Associated with these measure-theoretic analytical concepts, we define for a vector-valued measure $\mu$  the set 
\[
	\Theta^{s}(\mu) \coloneqq \set{x \in \R^n}{\limsup_{r \to 0} \frac{|\mu|(B_r(x))}{r^{s}} > 0},
\]
consisting of all its positive $s$-dimensional density points. 
The following results associate the approximate continuity and differentiability properties of $u$ to specific dimensional densities of $\Acal u$.
\begin{thmx}[Density estimates]\label{thm:fine}
Let $u \in L^1_\loc(\Omega;V)$ and assume that $\Acal u \in \Mcal(\Omega;W)$. Then, at $|\Acal u|$-almost every $x \in \Omega$ where $u$ is approximately Lebesgue continuous it holds
\[
	\limsup_{r \to 0^+} \frac{|\Acal u|(B_r(x))}{r^{n-1}} = 0.
\]
In particular,
\[
|\Acal u|(\Theta^{n-1}(\Acal u) \setminus S_u) = 0.
\]
\end{thmx}

A standard measure-theoretic result (cf.~\cite[Theorem 2.56]{APF}) yields the following structural result:

\begin{corollary}\label{cor:fine}
	If $u \in L^1_\loc(\Omega;V)$ is everywhere approximately Lebesgue continuous and $\Acal u \in \Mcal(\Omega;W)$, then $|\Acal u|$ vanishes on Borel sets that are $\sigma$-finite with respect to $\Hcal^{n-1}$, i.e.,
	\[
		\text{$U \subset \Omega$ Borel with $\Hcal^{n-1}(U) < \infty$} \quad \Longrightarrow \quad |\Acal u|(U) = 0.
	\]
In particular, 
\[
	|\Acal u|(M) = 0
\]
for all $C^1$-hyper-surfaces $M \subset \Omega$.
\end{corollary}

\begin{remark}There exist functions that are everywhere Lebesgue continuous and fail to be continuous on dense sets.
\end{remark}

\begin{proof} Let $F \subset \Omega$ be the set of points of $\Acal u$ where the density $\Acal u/|\Acal u|(x)$ exists and has norm one. Clearly, at all Lebesgue continuity points $x$ it holds 
\begin{equation}\label{eq:dos}
	\limsup_{r \to 0^+}  \left\{ \inf_{v_r \in V} \fint_{\partial B_r(x)} |u - v_r| = 0 \right\}\,.
\end{equation}
Indeed, the infimum on the right hand side can be chosen to be $v_r \equiv u^*(x)$. In general, the converse fails. We call such points the set of quasi Lebesgue points.

We give a contradiction argument: Let $x \in F$ be a quasi Lebesgue continuous point and let $I$ be the full-measure set of all $r \in (0,\infty)$ satisfying the following properties: $u^*$ exists on $\partial B_r$ and and $|\Acal u|(\partial B_r(x)) = 0$. Using that $\Abb$ has zero mean on the sphere, we deduce from Lemma~\ref{lem:GG} the estimate
\begin{align*}
	\frac{|\Acal u(B_r(x))|}{n\omega_n r^{n - 1}} & = \left|\frac ns\fint_{\partial B_r(x)}\Omega_\Abb[u^* - v_r] \, dS\right| \\
	&  \le \|\Abb\|_{C(\Sbf^{n-1})} \frac ns \fint_{\partial B_r(x)} |u - v_r| \, dS
\end{align*}
for all $r \in I$. Letting $r \to 0^+$ in $I$, we deduce from the existence of the density $\Acal u/|\Acal u|(x)$  that
\[
	\limsup_{\substack{r \in I\\r \to 0^+}}\frac{|\Acal u|(B_r(x))}{r^{n-1}} = 0.
\]
Now, let $s \in (0,\infty)$ and choose $r_s \in I$ with $s \le r_s \le 2s$ (here we are using that $I$ has full measure) so that
\[
	\frac{|\Acal u|(B_s(x))}{s^{n-1}}  \le 2^{n-1} \frac{|\Acal u|(B_{r_s}(x))}{r_s^{n-1}}.
\]
Since every infinitesimal sequence $\{s_j \to 0\}$ gives rise to an infinitesimal sequence $\{r_{s_j} \to 0\}$ with $r_{s_j} \in I$, the previous limsup estimate gives the desired 
\[
	\limsup_{s \to 0^+} \frac{|\Acal u|(B_s(x))}{s^{n-1}} = 0.
\]
Since $F$ is a full $|\Acal u|$-measure set of $\Omega$, this finishes the proof. 
\end{proof}

\subsection*{Acknowledgements} The author wishes to extend special thanks to Hidde Sch\"{o}nberger for fruitful discussions concerning the functional equivalence between the local and the nonlocal operators. In particular, for kindly pointing to me that the original statement was false as it was written, and for helping me reach the observation that no  Korn-type inequality may hold between the local and nonlocal operators (see Corollaries~\ref{cor:no_Korn} and~\ref{cor:no_Korn2}).
I would also like to thank the referee for their careful review, corrections, and comments that led to better presentation of this work. This work was supported by the F.R.S FNRS grant no. 40005112 and by the European Union (ERC, ConFine, 101078057). 

\subsection*{Disclaimer}  Views and opinions expressed are however those of the author(s) only and do not necessarily reflect those of the European Union or the European Research Council Executive Agency. Neither the European Union nor the granting authority can be held responsible for them.

\appendix

\section{A $\rho$-multiplier result}\label{sect:multi}Let $\rho : \R^n \to [0,\infty]$ be a radially symmetric probability function density. Associated to $\rho$, we define the profile function
\[
    g_\rho(r) \coloneqq n\frac{\hat \rho(r)}{r}\,, \qquad r \in [0,\infty)\,.
\]
Notice that 
\[
	\mu_\rho(x) \coloneqq \int_0^\infty g_\rho(r) 1_{B_r}(x) \, dr
\]
defines a radial, nonnegative, measurable probability function on $\R^n$. Indeed, by Fubini's theorem we get
\[
\int \mu_\rho \, dx = \int_0^\infty n\omega_n r^{n-1} \hat \rho(r) \, dr = \int \rho \, dx = 1\,.
\]
We also define the function of $n+2$ variables   
\[
	\Psi_\rho(Z) \coloneqq \frac{1}{2\pi} \cdot\frac{g_\rho(|Z|)}{|Z|}\,, \qquad Z \in \R^{n+2}\,,
\] 
which also defines a probability function $\R^{n+2}$ since, by the properties of the Gamma function it holds that (see~\cite[Section A.3]{grafakos2008classical})
\[
\forall k = 1,2,\dots, \qquad	\frac{(k+2)\omega_{k+2}}{k\omega_k} = \frac{2\pi}{k}\,,
\]
and hence
\begin{align*}
    \int_{\R^{n+2}} \Psi_\rho \, dz & 
    = \int_0^\infty (n+2)\omega_{n+2} r^{n+1} \hat {\Psi_\rho}(r) \, dr  \\
    & = \int_0^\infty n \omega_n r^{n-1}\hat \rho(r) \, dr = \int \mu_\rho = 1\,.
\end{align*}
    For $f \in \Scal(\R^n)$, we consider the following superposition of $r$-scale convolutions 
    \begin{align*}
        Tf & \coloneqq \int_0^\infty g_\rho(r) (1_{B_r} \star f) \, dr \,.
    \end{align*}

\begin{lemma}\label{lem:Psi} The Fourier transform of $\mu_\rho$ is given by
\begin{align*}
	\widehat{\mu_\rho}(\xi)  & \coloneqq  \frac{1}{|\xi|^{n/2}}  \int_0^\infty   nr^{\frac n2 -1} \hat \rho(r) J_{\frac n2}(2\pi r|\xi|) \, dr  
	= \widehat{\Psi_\rho}(\xi,0,0).
\end{align*}
In particular, $\widehat{\mu_\rho}$ is strictly positive if and only if $\widehat{\Psi_\rho}$ is strictly positive.
 
Furthermore, $\widehat{\mu_\rho}$ defines an $L^p$-multiplier for all $1 \le p \le \infty$ as
\[
Tf = f \star \mu_\rho\,\qquad \text{for all $f \in \mathcal S(\R^n)$\,,}
\] 
and it satisfies
	\[
	\widehat{\mu_\rho} \in C(\R^n) \quad \text{with} \quad |\widehat {\mu_\rho}(\xi)| \to 0 \; \text{as $|\xi| \to \infty$\,.}
	\]
\end{lemma}
\begin{proof} 
	Let us first derive the asserted expression for $\widehat{\mu_\rho}$.
	Since the distributional Fourier transform of $\mu_\rho$ commutes with the integral, it satisfies the identity 
	\begin{align*}
		\widehat{\mu_\rho}(\xi) & = \frac{1}{|\xi|^{\frac n2}} \int_0^\infty  r^{\frac n2} \hat{g_\rho}(r) J_{\frac n2}(2\pi r|\xi|) \, dr  \\
		&=  \frac{2\pi}{|\xi|^{\frac{n}2}} \int_0^\infty  r^{\frac{n}{2}+1}  \hat{\Psi_\rho}(r) J_{\frac{n}2}(2\pi r|\xi|) \, dr  \\
		& = \widehat \Psi_\rho(\xi,0,0)\,.
	\end{align*}
	Here, we have used the properties of Bessel functions along their formula to represent the Fourier transform of radial functions including the unit ball (see~\cite[B.5]{Grafakos}). 
Being the Fourier transform of a probability function, $\widehat{\mu_\rho}$ defines an $L^p$ multiplier by convolution that coincides with $T$ on Schwartz functions. In particular, by Young's inequality it follows that
\[
	\|Tf\|_{L^p} \le \|f\|_{L^p}\qquad \text{for all $1\le p\le \infty$.}
\]
	
	The fact that $\mu_\rho \in L^1(\R^n)$ implies that $\widehat{\mu_\rho}$ is continuous. Finally, the Riemann--Lebesgue lemma 
	guarantees that $|\widehat{\mu_\rho}(\xi)|$ decays to zero as $|\xi|$ tends to infinity (so, in fact, $\widehat{\mu_\rho}$ is uniformly continuous). 
\end{proof}

A direct consequence of this result is that $(\widehat {\mu_\rho})^{-1}$ does not define an $L^2$-multiplier. In particular, we have the following negative result:

\begin{corollary}\label{cor:non_invertible} Let $1 \le p \le \infty$. 
There exists no constant $C$ such that
\[
    \|f\|_{L^p} \le C \|Tu\|_{L^p}
\]
for all $f \in \mathscr D(\R^n)$. 
\end{corollary}
\begin{proof}
    By the decay of $|\widehat{\mu_\rho}(\xi)|$ at infinity (see the lemma above), it follows that $(\mu_\rho)^{-1} \notin L^\infty(\R^n)$. The desired assertion then follows from the well-known fact (see, e.g.,~\cite[Section~2.5.4]{Grafakos}) that the space of $L^p$ multipliers is a subset of $L^\infty$.  
\end{proof}

In spite of this negative result, we have a few useful observations. First, we get the following characterizations of injectivity:

\begin{corollary}[Full-space]\label{thm:injective} 
  Let $1 \le p \le 2$. The following are equivalent:
    \begin{enumerate}[(a)]
        \item $T$ is one-to-one in $L^p(\R^n)$.
        \item $\widehat{\mu_\rho}(\xi) \neq 0$ for almost every $\xi \in \R^n$. 
    \end{enumerate}
\end{corollary}

    \begin{corollary}[Torus]\label{thm:injective} 
 Let $1 \le p \le \infty$. The following are equivalent:
    \begin{enumerate}[(a)]
        \item $T$ is one-to-one in $L^p(\Tbb^n)$.
        \item $\widehat{\mu_\rho}(m) \neq 0$ for all $m \in \Zbb^n - \{0\}$.
    \end{enumerate}
\end{corollary}

\begin{proof}[Proof (of both statements).]  
First, we address the full-space case. Let $ u \in L^p(\R^n)$. Since both $\widehat u$ and $\widehat{\mu_\rho}$ are locally integrable functions, it follows from Fourier inversion (for tempered distributions) that $T$ is injective if and only if $\widehat{\mu_\rho}(\xi) \neq 0$ for almost every $\xi \in \R^n$. Since $\widehat{\mu_\rho}$ is continuous, the latter condition is equivalent with requiring $\widehat{\mu_\rho} \neq 0$ almost everywhere. 

The proof of the second corollary is similar, except for a few details that are worthwhile to mention. Firstly, the compactness of the torus and a standard regularization argument allows one to reduce the proof to the case $p=1$. Secondly, the condition is restricted to the lattice of points $m \in \Zbb^n$ because one uses Fourier coefficients instead of the Fourier transform. And lastly, the non-vanishing condition at the frequency  $m=0$ is satisfied. Indeed, recalling that $\widehat{\mu_\rho}$ is a probability density, we get $\widehat{\mu_\rho}(0) = 1$ (regardless of the choice of $\rho$). 

This finishes the proof. 
\end{proof}

\section{A measure-theoretic lemma}

\begin{lemma}\label{lem:Leb}
Let $\mu \in \Mcal^+(\R^n)$ be a finite measure and let $s > 0$ be given. Then, the set
\[
	 R_\mu \coloneqq \{x \in \R^n : \mu(\partial B_s(x)) > 0\}
\]
is negligible with respect to the $n$-dimensional Lebesgue measure.
\end{lemma}
\begin{proof} Since $\mu$ is a Borel measure, it follows that the map
\[
	g(x)\coloneqq \mu(B_s(x)) = \int \chi_{B_s(x)} \,  d\mu(y) \,,\qquad x \in \R^n
\]
is Borel measurable. Moreover, by Young's convolution inequality, it follows that 
\[
	\|g\|_{L^1} \le \|\eta_s \star \mu\|_{L^1} \le \|\eta_s\|_{L^1} |\mu|(\R^n) < \infty, \qquad \eta_s \coloneqq \eta \star \frac{\chi_{B_s}}{|B_s|}\,,
\]
for any nonnegative test function $\eta \in C^\infty_c(B_{2s})$ such that $\eta \equiv 1$ on $B_s$. In particular, $g$ is integrable; hence, it is Lebesgue continuous at almost every $x \in \R^n$. Notice however that if $|\mu|(\partial B_s(x)) > 0$, then there exists a real $\delta >0$ and an open cone $\mathscr C \subset \R^n$  such that 
\[	
	\frac{\mathscr L^n(\mathscr C \cap B_1)}{|B_1|} \ge \delta, \qquad g(y) \ge g(x) + \delta  \quad \text{for every $y \in (B_\delta \cap \mathscr C) + x$}.
\]
In particular, by a scaling argument, it follows that
\[
	\limsup_{r \to 0} \fint_{B_r(x)} |g - g(x)| \ge \delta^2 > 0 \qquad \text{for all $r > 0$}.
\]
From this basic analysis, we deduce that $R_\mu \subset S_g$ ---where $S_g$ is the set of Lebesgue discontinuous points of $g$. The Lebesgue differentiation theorem gives 
\[
	\mathscr L^n(R_\mu) \le \mathscr L^n(S_g) = 0,
\]
which is what we wanted to show.
\end{proof}

\section{Area-convergence and $L^1$-compactness}

The next lemma shows that area convergence of $L^1$ sequences to an $L^1$ limit is equivalent to strong $L^1$-convergence.

\begin{lemma}\label{lem:area_discriminates}
Let $(u_j)\subset L^1(\R^n;V)$ and $u \in L^1(\R^n;V)$. The following are equivalent:
\begin{enumerate}[(i)]
\item $u_j\,\Lcal^n \longrightarrow u \, \Lcal^n$ in the area sense of measures
\item $u_j \longrightarrow u$ in $L^1$
\end{enumerate}
\end{lemma}
\begin{proof}
We prove $(i) \to (ii)$. The fact that $u_j \, \Lcal^d$ converges in area to an absolutely continuous limit $u\,\Lcal^d$ has two important consequences. Firstly, it implies that $(u_j)_j$ is a tight family of measures. Secondly,  through~\cite[Remark~2.6]{advances} and~\cite[Remark~2.10(b)]{AB}), it implies that $(u_j)_j$ is pre-compact with respect to sequential weak $L^1$-convergence. In particular, by the Dunford--Pettis criterion, the sequence $(u_j)_j$ is equi-integrable on $\R^d$. On the other hand, now appealing to~\cite[Remark~2.6]{advances} and~\cite[Remark~2.10(a)]{AB}), we deduce that $u_j \to  u$ in measure on $\R^n$. This shows that the sequence $(u_j)_j$ satisfies the three assumptions of the Vitali convergence theorem on $\R^n$, and hence, $u_j \to u$ in $L^1(\R^d)$ as desired. 

We now prove that $(ii) \to (i)$. By definition, $u_j \, \mathcal L^n \toweakstar u \mathcal L^n$ as measures. Moreover, the Lipschitz constant $\textrm{Lip}(g)$ of the area integrand $g(v) = \sqrt{1 + v^2}$ is one and therefore $|g(v_1) - g(v_2)|\le |v_1 - v_2|$ for all $v_1,v_2 \in V$. In particular, 
\[
	\limsup_{j \to \infty} \left| \int g(u_j) - g(u) \, dx \right| \limsup_{j \to\infty} \le \int |u_j - u| \, dx\, = 0\,.
\]
This finishes the proof.
\end{proof}

\end{document}